\RequirePackage[l2tabu, orthodox]{nag}
\documentclass[11pt]{amsart} 

\author{Andy Hammerlindl}
\address{School of Mathematics, Monash University, Victoria 3800 Australia}  \email{andy.hammerlindl@monash.edu}
\author{Rafael Potrie}
\address{Centro de Matem\'atica, Universidad de la Rep\'ublica, Uruguay \& IRL-IFUMI (CNRS) }
\email{rpotrie@cmat.edu.uy}
\title{Horizontality of partially hyperbolic foliations}

\usepackage{amsmath}
\usepackage{amssymb}
\usepackage{amsfonts}
\usepackage{amsthm}
\usepackage{amscd}
\usepackage{graphicx}
\usepackage{setspace}
\usepackage{enumitem}
\usepackage{cleveref}
\usepackage{multicol}
\usepackage{fourier}

\usepackage[T1]{fontenc}

    \newcommand{\bbR}{\mathbb{R}}
    
    \newcommand{\bbH}{\mathbb{H}}
    
    \newcommand{\bbZ}{\mathbb{Z}}

    \newcommand{\subof}{\subset}
    \newcommand{\ti}{\times}
    \newcommand{\sans}{\setminus}

    \newcommand{\Es}{E^s}
    \newcommand{\Ec}{E^c}
    \newcommand{\Eu}{E^u}
    \newcommand{\Ecu}{E^{cu}}
    \newcommand{\Ecs}{E^{cs}}

    \newcommand{\Ws}{W^s}
    
    \newcommand{\Wu}{W^u}

    \newcommand{\inv}{^{-1}}
    \newcommand{\invn}{^{-n}}

    \newcommand{\Fcal}{\mathcal{F}}
    \newcommand{\Gcal}{\mathcal{G}}
    \newcommand{\Fs}{\mathcal{F}^s}
    \newcommand{\Fc}{\mathcal{F}^c}
    \newcommand{\Fu}{\mathcal{F}^u}
    \newcommand{\Fcu}{\mathcal{F}^{cu}}
    \newcommand{\Fcs}{\mathcal{F}^{cs}}

    \newcommand{\Fchat}{{\mathcal{\hat F}}^c}
    \newcommand{\Fcshat}{{\mathcal{\hat F}}^{cs}}

    \newcommand{\Lchat}{{\hat L}^c}
    \newcommand{\Lcshat}{{\hat L}^{cs}}
    \newcommand{\Lcuhat}{{\hat L}^{cu}}

    \newcommand{\hcs}{h^{cs}}
    \newcommand{\hcu}{h^{cu}}

    \newcommand{\hcsep}{h^{cs}_{\ep}}
    \newcommand{\hcuep}{h^{cu}_{\ep}}
    \newcommand{\hcep}{h^{c}_{\ep}}

    \newcommand{\Lep}{L_{\ep}}
    \newcommand{\Lcsep}{L^{cs}_{\ep}}
    \newcommand{\Lcuep}{L^{cu}_{\ep}}
    \newcommand{\Lcep}{L^{c}_{\ep}}

    \newcommand{\Fcsep}{\mathcal{F}^{cs}_{\ep}}
    \newcommand{\Fcuep}{\mathcal{F}^{cu}_{\ep}}
    \newcommand{\Fcep}{\mathcal{F}^{c}_{\ep}}

    \newcommand{\Wbran}{\mathcal{W}}
    \newcommand{\Wcsbran}{\mathcal{W}^{cs}}
    \newcommand{\Wcubran}{\mathcal{W}^{cu}}
    \newcommand{\Wcbran}{\mathcal{W}^{c}}

    \newcommand{\Lc}{L^{c}}
    \newcommand{\Ls}{L^{s}}
    \newcommand{\Lu}{L^{u}}
    \newcommand{\Lcs}{L^{cs}}
    \newcommand{\Lcu}{L^{cu}}

    \newcommand{\Jc}{J^{c}}
    \newcommand{\Js}{J^{s}}
    \newcommand{\Ju}{J^{u}}

    \newcommand{\tM}{\widetilde{M}}

    \newcommand{\Mhat}{\hat{M}}
    
    \newcommand{\fhat}{\hat{f}}

    \newcommand{\ep}{\epsilon}
    \newcommand{\lam}{\lambda}
    \newcommand{\Lam}{\Lambda}
    \newcommand{\Lamcs}{\Lambda^{cs}}
    \newcommand{\Lamcu}{\Lambda^{cu}}

    \newcommand{\Gam}{\Gamma}
    \newcommand{\gam}{\gamma}
    \newcommand{\Sig}{\Sigma}
    \newcommand{\sig}{\sigma}

    \newcommand{\al}{\alpha}
    
    \newcommand{\bt}{\beta}

    \newcommand{\phic}{\phi}
    \newcommand{\hatphic}{\hat \phi}

    \newcommand{\qandq}{\quad \text{and} \quad}
    \newcommand{\qorq}{\quad \text{or} \quad}

    \newcommand{\dist}{\operatorname{dist}}
    \newcommand{\area}{\operatorname{area}}
    \newcommand{\length}{\operatorname{length}}
    \newcommand{\volume}{\operatorname{volume}}
    \newcommand{\diam}{\operatorname{diam}}

    \newcommand{\ior}{\operatorname{int}}

    \newcommand{\del}{\partial}

\newcommand{\simpleCD}[9]{
\begin{CD}
#1 @>#2>> #3 \\
@VV#4V  @VV#6V \\
#7 @>#8>> #9
\end{CD}
}

\numberwithin{equation}{section}
\newtheorem{thm}[equation]{Theorem}
\newtheorem{cor}[equation]{Corollary}
\newtheorem{lemma}[equation]{Lemma}
\newtheorem{claim}[equation]{Claim}
\newtheorem{prop}[equation]{Proposition}
\newtheorem{addendum}[equation]{Addendum}

\theoremstyle{remark}
\newtheorem*{remark} {\textbf{Remark}}

\crefname{figure}{figure}{Figure}

\begin{document}

%

\begin{abstract}
    We show exactly which Seifert manifolds support
    partially hyperbolic dynamical systems.
    In particular, a circle bundle over a higher-genus surface
    supports
    a partially hyperbolic system if and only
    if it supports an Anosov flow.
    We also show for these systems that the
    center-stable and center-unstable foliations
    can be isotoped so that their leaves are transverse
    to the circle fibering.
    As a consequence,
    every partially hyperbolic system defined on the unit tangent bundle
    of a higher-genus surface is a collapsed Anosov flow.
\end{abstract}

\maketitle



\section{Introduction} \label{sec:intro}

Partially hyperbolic systems have been studied since the 1970s
as a large and important class of chaotic dynamical systems
and as a natural generalization of the concept of a
uniformly hyperbolic system.
We give a precise definition of partial hyperbolicity
in \cref{sec:bran}.

Despite decades of intensive study,
it was a long-standing open question if any 3-dimensional manifold could
support a partially hyperbolic system.
This question was finally answered in the negative
by Brin, Burago, and Ivanov, who developed
branching foliation theory and showed that the 3-sphere
does not support a partially hyperbolic system
\cite{BBI1, BI, BBI2}.
In fact, they showed that a 3-manifold with a partially hyperbolic
system must also support a Reebless foliation
and this rules out many 3-manifolds.
This breakthrough led in the subsequent years to a wealth of
new results for partially hyperbolic systems.
See the surveys \cite{hp2018survey, crhrhu2018survey}
as well as the more recent results in
\cite{bffp1, bffp2, bfp2023collapsed}.

We now answer the question for Seifert manifolds,
also called Seifert fiber spaces,
showing exactly which of these manifolds support partially
hyperbolic diffeomorphisms.

\begin{thm} \label{thm:whichseifert}
    A Seifert manifold $M$ supports a partially hyperbolic diffeomorphism
    if and only if
    \begin{enumerate}
        \item $M$ is a nilmanifold,
        \item
        $M$ is double covered by a nilmanifold, or
        \item
        $M$ finitely covers the unit tangent bundle of a
        hyperbolic orbifold.
    \end{enumerate} \end{thm}
Here, the 3-dimensional nilmanifolds include the 3-torus.

\begin{cor} \label{cor:productmfld}
    Let $S$ be a closed oriented surface of genus $g \ge 2,$
    then there are no partially hyperbolic diffeomorphisms
    on $S \ti S^1.$
\end{cor}
Here as usual, $S^1$ denotes the circle.

This result improves and builds on our previous paper with Shannon
\cite{hps2018seif}. In that paper, we obtained the same results under
additional assumptions on the partially hyperbolic diffeomorphism, such as
transitivity or information on the action in the fundamental group. In fact,
those assumptions were used to prove the main technical tool in  order to get
the results, which is to deduce horizontality of the dynamical foliations. Our
main result here is to prove this general result (which is interesting on its
own) without needing further assumptions:  

\begin{thm} \label{thm:main}
    For a partially hyperbolic diffeomorphism defined on a
    circle bundle over a higher-genus surface,
    the cs and cu branching foliations are horizontal.
\end{thm}
Throughout this paper, a \emph{higher-genus surface} is
a closed oriented surface $S$ without boundary and
with genus at least two.
We further assume that $S$ is equipped with some specific choice of hyperbolic
metric and so is a quotient of the hyperbolic plane $\bbH^2.$

In \cite{bfp2023collapsed} a class of partially hyperbolic
diffeomorphisms is proposed,
modeled using self-orbit equivalences of Anosov flows.
This class, called \emph{collapsed Anosov flows}, could potentially be a
model to describe \emph{all} partially hyperbolic diffeomorphisms on 3-manifolds
without a virtually solvable fundamental group.
For the virtually solvable case, a complete classification
is given in \cite{HP2}.
We do not give the definition of a collapsed Anosov flow here
as it is somewhat technical, and refer the reader to
\cite{bfp2023collapsed}.

In \cite{fp2026transverse}, it is shown that under the
assumption of horizontality of the branching foliations, partially hyperbolic
diffeomorphisms of unit tangent bundles of higher genus surfaces must be
collapsed Anosov flows.
\Cref{thm:main} shows that this assumption can be removed
and so we obtain the following:

\begin{cor} \label{cor:collapsed}
    Every partially hyperbolic diffeomorphism
    on the unit tangent bundle of a higher-genus surface
    is a collapsed Anosov flow.
\end{cor}
This extends the results of \cite{bffp1,bffp2,fp2024pseudo} that establish
this result for hyperbolic 3-manifolds as well as some isotopy classes of
Seifert manifolds.
Note that \cite{BGHP} shows that there are plenty of isotopy classes of
diffeomorphisms of unit tangent bundles admitting partially hyperbolic
diffeomorphisms. 

\smallskip

\Cref{thm:main} holds if we replace
``circle bundle over a higher-genus surface''
with
``Seifert manifold over a hyperbolic orbifold''
since the latter condition can always be replaced by the former condition
by lifting to a finite cover.
Lifting to a finite cover does not affect horizontality of the foliations.
Using the results given in sections 6 and 7
of \cite{hps2018seif}, \cref{thm:main} implies the following:

\begin{thm} \label{thm:onlyflow}
    Suppose $M$ is a Seifert manifold over a hyperbolic orbifold $\Sig.$
    Then $M$ supports a partially hyperbolic diffeomorphism
    if and only if
    it finitely covers the unit tangent bundle of $\Sig.$
\end{thm}
If the base orbifold is not hyperbolic,
then the Seifert manifold has virtually nilpotent fundamental group
\cite{sco1983geometries}.
As mentioned above,
the partially hyperbolic systems on such manifolds
have already been classified.
Combining the results in \cite[Appendix A]{HP2}
with \cref{thm:onlyflow} yields the proof of
\cref{thm:whichseifert}.

The proof of \cref{thm:main} relies on the following result of
Brittenham \cite{britt1993essential}.

\begin{thm}
    [Brittenham] \label{thm:brittenham}
    Let $M$ be a circle bundle over a higher-genus surface $S$
    and let $\Fcal$ be a foliation on $M$ without compact leaves
    and tangent to a continuous plane field.
    Then there is a homeomorphism $h : M \to M$ isotopic to
    the identity such that
    $h(\Fcal)$ is a $C^0$ foliation with $C^1$ leaves
    tangent to a continuous plane field
    and 
    every leaf of $h(\Fcal)$ is either
    a union of fibers or is transverse to the fibering.
\end{thm}

The theorem as originally given in \cite{britt1993essential}
is stated in a slightly weaker form than the above.
However, Brittenham clarifies in the introduction
to \cite{britt1999boundary}
that a version equivalent to \cref{thm:brittenham} above holds.
The assumption of no compact leaves above rules out the Reeb components
discussed in \cite{britt1997bundle, britt1999boundary}.
See also section 4.10 of \cite{calegari2007book} for another explanation
of how \cref{thm:brittenham} can be proved.

Let $\Fcal$ be a foliation which satisfies the
hypotheses of \cref{thm:brittenham}.
A leaf of $\Fcal$ is \emph{vertical} if it contains
a loop which is freely homotopic to a fiber of the circle bundle.
A leaf is \emph{horizontal} if it is not vertical.
The function $h$ in \cref{thm:brittenham} takes each vertical leaf to a union of
fibers
and each horizontal leaf to a surface transverse to the fibers.
We call the foliation $\Fcal$ \emph{horizontal} if every leaf is horizontal.

In most of this paper, we will want to further isotope
a foliation to put the vertical leaves into a slightly nicer form.
We say that a $C^0$ foliation $\Fcal$ with $C^1$ leaves is in \emph{ideal position} if
every horizontal leaf is transverse to the fibering and
every vertical leaf is the pre-image of a geodesic on $S.$

\begin{thm} \label{thm:ideal}
    Let $M$ be a circle bundle over a higher-genus surface $S$
    and let $\Fcal$ be a foliation on $M$ without compact leaves
    and tangent to a continuous plane field.
    Further assume that no two distinct vertical leaves are
    at finite Hausdorff distance when lifted to the universal cover.
    Then there is a homeomorphism $h : M \to M$ isotopic to the identity
    such that $h(\Fcal)$ is in ideal position.

    Moreover, $h(\Fcal)$ is tangent to a plane field
    which is continuous on the set formed by the union all horizontal leaves.
\end{thm}
Note here that the map $h$ in \cref{thm:ideal} is only a homeomorphism
and not a diffeomorphism. This difference in regularity is discussed
in \cref{sec:regularity}.
We give the proof of \cref{thm:ideal} in \cref{sec:ideal}.

For an embedded circle $C$ in a circle bundle $M$ over a surface $S,$
we call $C$ a \emph{vertical circle} if it is isotopic
to a fiber of the circle bundle.
Much of the proof of \cref{thm:main} is not specific to partially hyperbolic
dynamics and holds for any two transverse foliations on $M$
which satisfy certain properties. This non-dynamical part of the proof is
encapsulated in the following theorem.

\begin{thm} \label{thm:twofolns}
    Let $M$ be a circle bundle over a higher-genus surface $S$
    and suppose on $M$ that there are two transverse foliations
    $\Fcs$ and $\Fcu$ intersecting in a one-dimensional foliation $\Fc$
    and which satisfy the following conditions:
    \begin{enumerate}
        \item each of the two foliations is $C^0$ regular with $C^1$ leaves
        and is tangent to a continuous plane field on $M;$
        \item
        neither foliation has compact leaves;
        \item
        no two distinct vertical leaves of $\Fcs$ or of $\Fcu$ are
        at finite Hausdorff distance when lifted to the universal cover;
        \item
        \label{item:firstcircle}
        the foliation $\Fc$ contains at most finitely many vertical circles;
        \item
        \label{item:secondcircle}
        a leaf of $\Fc$ is a vertical circle if and only if
        it lies in both a vertical leaf of $\Fcs$ and
        a vertical leaf of $\Fcu;$
        \item
        \label{item:lastcircle}
        each vertical leaf of $\Fcs$ or $\Fcu$
        contains at most one vertical circle of $\Fc;$ and
        \item
        \label{item:nooneended}
        inside a vertical leaf of $\Fcs$ or $\Fcu,$
        the $\Fc$ foliation does not contain any ``one-ended leaves''.
    \end{enumerate}
    Then both $\Fcs$ and $\Fcu$ are horizontal.
\end{thm}
Of course, the conclusion of \cref{thm:twofolns} means that there are no vertical
leaves or vertical circles and so items (3) though (7) of the list are
vacuously true. We give a precise definition of a ``one-ended leaf'' in
\cref{sec:folncyl}.

For the dynamical part of the proof of \cref{thm:main}, we show the following:

\begin{prop} \label{prop:approxgood}
    Let $f$ be a partially hyperbolic diffeomorphism on
    a circle bundle over a higher-genus surface
    and such that $Df$ preserves the orientations of
    $\Eu, \Ec,$ and $\Es.$
    Then the cs and cu branching foliations have
    approximating true foliations which satisfy the
    hypotheses of \cref{thm:twofolns}.
\end{prop}

The definitions of branching foliations and their approximating true
foliations are given in \cref{sec:bran}.
The above definitions for horizontal and vertical leaves also hold for branching
foliations in the setting of \cref{thm:main}.
\Cref{prop:approxgood} and \cref{thm:twofolns} together imply \cref{thm:main}.

The proof of \cref{thm:twofolns} splits into two cases depending on whether or not
$\Fc$ has any vertical circles.
The case without vertical circles is simpler and given in \cref{sec:average}.
The case with vertical circles is much more complicated and
all of sections \ref{sec:regularity} to \ref{sec:finale}
contribute to its proof.

\medskip

\noindent\textbf{Acknowledgements.}
The authors thank Sergio Fenley and Mario Shannon for helpful discussions.
This research was partially funded by the Australian Research Council.
R.P.~was partially supported by CSIC.

\section{Branching foliations} \label{sec:bran} 

In this section, we go into further detail about the branching foliations
and their approximating true foliations as appearing in \cref{prop:approxgood}.

For a continuous plane field $E$ on a 3-manifold $M,$
a \emph{branching foliation} $\Wbran$ is a collection of immersed surfaces,
called \emph{leaves}, with the following properties:
\begin{enumerate}
    \item every leaf is a boundaryless and complete surface tangent to $E,$
    \item
    for every point $x \in M,$ at least one leaf passes through $x,$
    \item
    no two leaves topologically cross
    and no leaf topologically crosses itself, and
    \item
    if a sequence $\{L_n\}$ of leaves
    converges in the compact-open
    topology to a limit surface $L,$ then
    $L$ is also a leaf.
\end{enumerate}
For a diffeomorphism $f : M \to M,$ we say that the branching foliation
is $f$-\emph{invariant} if $L \in \Wbran$ if and only if $f(L) \in \Wbran.$

We now consider partially hyperbolic systems.
A $C^1$ diffeomorphism $f : M \to M$ of a closed Riemannian 3-manifold
is \emph{partially hyperbolic} if there is $n \ge 1$ and a $Df$-invariant
splitting of the tangent bundle into one-dimensional subbundles
$TM = \Es \oplus \Ec \oplus \Eu$
such that
\[
    \| Df^n v^s \| < 1 < \| Df^n v^u \| 
    \qandq
    \| Df^n v^s \| < \| Df^n v^c \| < \| Df^n v^u \|
\]
for all points $x \in M$ and unit vectors
$v^s \in \Es_x,$ $v^c \in \Ec_x$ and $v^u \in \Eu_x.$
By changing the metric on $M,$ we may freely assume that $n = 1$
and we do so for the rest of the paper.
By classical results, the stable $\Es$ and unstable $\Eu$ bundles
are uniquely integrable and yield foliations $W^u$ and $W^s.$
In what follows, we write $\Ecs = \Ec \oplus \Es$ and $\Ecu = \Ec \oplus \Eu.$

\begin{thm}
    [Burago--Ivanov] \label{thm:bran}
    Suppose $f : M \to M$ is
    a partially hyperbolic diffeomorphism of a closed 3-manifold $M$
    and that the derivative $Df$ preserves orientations of each
    of $\Eu, \Ec,$ and $\Es.$
    Then there exists an $f$-invariant branching foliation $\Wcsbran$
    tangent to $\Ecs.$
\end{thm}
For the remainder of the section, we assume $f$ and $\Wcsbran$ are as in this
theorem. There are also analogous results for a center-unstable
branching foliation $\Wcubran$ tangent to $\Ecu.$
Burago and Ivanov further prove the following.

\begin{thm}
    [Burago--Ivanov] \label{thm:approxfoln}
    Let $\ep > 0.$
    Then there exists a foliation $\Fcsep$ and a continuous map
    $\hcsep : M \to M$ such that
    \begin{enumerate}
        \item $\Fcsep$ is a $C^0$ foliation with $C^1$ leaves and
        is tangent to a distribution which is $\ep$-close to $\Ecs,$
        \item
        for all $x \in M, d(\hcsep(x),x) < \ep,$
        \item
        for every leaf $\Lep$ in $\Fcsep,$ the image
        $L = \hcsep(\Lep)$ is a leaf in $\Wcsbran$
        and the restricted map $\hcsep|_{\Lep} : \Lep \to L$
        is a $C^1$-diffeomorphism,
        and
        \item
        for every leaf $L$ in $\Wcsbran,$
        there is at least one leaf $\Lep$ in $\Fcsep$ such
        that
        \[
            \hcsep(\Lep) = L.
        \] \end{enumerate} \end{thm}
We call $\Fcsep$ the \emph{approximating foliation} for $\Wcsbran$
and we call $\hcsep$ the \emph{collapsing map}.

Both of the above theorems are proved in \cite{BI}.
In that paper, they allow the possibility that two leaves of the 
branching foliation are the same up to reparameterization.
That is, they allow for immersed surfaces $i_1 : S_1 \to M$
and $i_2 : S_2 \to M$ both in $\Wcsbran$ such that there is a homeomorphism
$h_{12} : S_2 \to S_1$ with $i_1 \circ h_{12} = i_2.$
In the current paper, we do not allow duplicate surfaces.
We require that every surface in the branching foliation is unique,
even up to reparameterization. We then get the following property:
\begin{quote}
    for every leaf $L \in \Wcsbran,$
    there is a \emph{unique} leaf $\Lep \in \Fcsep$ such that $\hcsep(\Lep) = L.$
\end{quote}
We call $\Lep$ the \emph{approximating leaf} for $L.$
For more details on why this uniqueness can be assumed, see 
\cite[Appendix A.1]{bfp2023collapsed}.

The leaves of the branching foliations have additional nice properties
when lifted to the universal cover of $M.$
Let $L$ be the lift of a leaf in $\Wcsbran$ to the universal cover $\tM.$
Then
\begin{itemize}
    \item $L$ is an embedded surface in $\tM$ and is homeomorphic to $\bbR^2,$ and
    \item
    every lifted unstable leaf in $\tM$ intersects $L$ at most once.
\end{itemize}
These facts are proved in \cite{BI}.
See also sections 4 and 5 of the survey paper \cite{hp2018survey}.
These properties give us a way to define a branching center foliation.
Suppose $\Lcs$ is the lift of a leaf of $\Wcsbran$ to $\tM$
and $\Lcu$ is the lift of a leaf of $\Wcubran$ to $\tM.$
Then their intersection consists of a disjoint union of properly embedded lines
in $\tM$ tangent to the (lifted) center direction.
Let $\widetilde{\Wcbran}$ be the collection of all connected components
of intersections $\Lcs \cap \Lcu$ of this form.
Define $\Wcbran$ as the quotient of each curve in $\widetilde{\Wcbran}$
down to a curve in $M.$ Each element of $\Wcbran$ is an immersed curve
(either a circle or a copy of $\bbR$) tangent to $\Ec.$
More details on the center branching foliation are given in
\cite[\S 3.3]{bffp2}.

If $\ep > 0$ is sufficiently small, then leaves of the approximating foliations
$\Fcsep$ and $\Fcuep$ are transverse to each other and so their intersections
define a one-dimensional foliation $\Fcep.$
We will always assume that $\ep > 0$ is small enough that this property holds.
We now show that the leaves of $\Fcep$ approximate the leaves of $\Wcbran$
in the following sense.

\begin{prop} \label{prop:centerapprox}
    There is a continuous map $\hcep : M \to M$ such that
    \begin{enumerate}
        \item for each leaf $\Lep \in \Fcep,$
        the image $L = \hcep(\Lep)$ is a leaf of $\Wcbran$ and the restriction
        of $\hcep$ to $\Lep$ gives a covering map from $\Lep$ to $L;$
        \item
        for each leaf $L \in \Wcbran,$
        there is at least one leaf $\Lep \in \Fcep$ such that
        $\hcep(\Lep) = L.$
    \end{enumerate} \end{prop}
\begin{proof}
    We will construct a map $h : \tM \to \tM$ on the universal cover
    which is equivariant with respect to deck transformations.
    It will quotient down to the desired map on $M.$

    Consider surfaces $\Lcs, \Lcu \subof \tM$ which are
    lifts of leaves in $\Wcsbran$ and $\Wcubran$ respectively.
    Let $\Lcsep$ and $\Lcuep$ be the nearby lifts to $\tM$ of the
    corresponding approximating leaves in $\Fcsep$ and $\Fcuep.$
    In particular, $\Lcs$ and $\Lcsep$ are $\ep$-close in Hausdorff distance
    as subsets of $\tM,$ and the same holds for $\Lcu$ and $\Lcuep.$
    By a slight abuse of notation, we write $\hcsep : \tM \to \tM$
    for the lift of the collapsing map to the universal cover.
    In particular, $\hcsep(\Lcsep) = \Lcs.$

    Consider a point $x \in \Lcsep \cap \Lcuep$ and define $y = \hcsep(x) \in \Lcs.$
    Then $d(x,y) < \ep$ implies that $\dist(y, \Lcu) < 2 \ep.$
    We assume $\ep$ is small enough that the stable manifold $\Ws(y)$
    through $y$ intersects $\Lcu.$
    Therefore define a map $h : \Lcsep \cap \Lcuep \to \Lcs \cap \Lcu,$
    by setting $h(x)$ to be the intersection of $\Ws(y)$ with $\Lcu.$

    We show that $h$ is invertible.
    For this, note that the set $Y = \hcsep(\Lcsep \cap \Lcuep)$
    is a union of $C^1$ curves
    in $\Lcs$ whose tangents are $\ep$-close to the center direction.
    Moreover, inside of $\Lcs,$
    the set $Y$ is close in Hausdorff distance
    to $\Lcs \cap \Lcu.$
    The inverse map is therefore given by $h \inv(z) = (\hcsep) \inv(y)$ where $y$ is
    the intersection of $\Ws(z)$ with $Y.$
    This shows that $h : \Lcsep \cap \Lcuep \to \Lcs \cap \Lcu$
    is a homeomorphism.
    To extend $h$ to a map defined on all of $\tM,$
    consider all such intersections of lifts of pairs of leaves.
\end{proof}
By making $\ep > 0$ small, we can make the map $\hcep$ arbitrarily $C^0$-close
to the identity map.

We call a leaf of $\Wcbran$ a \emph{center circle} if it is the image
under $\hcep$ of a circle in $\Fcep.$
Note that being a center circle is a stronger property
than being an immersed circle $C \subof M$ tangent to the center
bundle $\Ec.$
Further, the uniqueness of approximating leaves in $\Fcsep$ and $\Fcuep$
does not imply uniqueness of approximating leaves in $\Fcep.$
A priori, it could be the case that multiple circular leaves in
$\Fcep$ all collapse down to the same center circle in $\Wcbran.$
However, we will show in \cref{sec:wellplaced} that
for partially hyperbolic systems defined on circle bundles,
distinct center circles in $\Wcbran$ map cannot have the
same image under $\hcep.$

\section{Foliations on cylinders} \label{sec:folncyl} 

In this section, we define the notion of a one-ended leaf
and also establish some properties of foliations on cylinders which
will be needed later.

Consider a continuous foliation $\Fcal$ on a cylinder.
Up to homeomorphism, we can suppose that the cylinder is $\bbR \ti S^1.$
We assume that the foliation is oriented in the sense
that there is a continuous flow $\phi$ on $\bbR \ti S^1$ such that
the leaves of the foliation are the orbits of the flow.
Because of this, we will use the words ``leaf'' and ``orbit''
interchangeably in what follows.

A \emph{half cylinder} is a subset of $\bbR \ti S^1$
of the form $(-\infty, a] \ti S^1$ or $[a, +\infty) \ti S^1.$
A leaf of $\Fcal$ is a \emph{one-ended leaf} if it is properly embedded
and contained in a half cylinder.
A leaf of $\Fcal$ is a \emph{two-ended leaf} if it is properly embedded
and not contained in any half cylinder.
A leaf of $\Fcal$ is a \emph{vertical circle} if it is 
isotopic to $\{0\} \ti S^1.$

\begin{thm} \label{thm:cylinderfoliation}
    Consider a continuous flow without fixed points on a cylinder.
    \begin{enumerate}
        \item If an orbit is properly embedded, then it is
        one of the following: \\
        a vertical circle,
        a one-ended leaf,
        or a two-ended leaf.
        \item
        The $\omega$-limit set of an orbit is one of the following: \\
        an empty set,
        a single vertical circle,
        or a union of one-ended leaves.
        \item
        If both the $\alpha$ and $\omega$-limit sets of an orbit
        are vertical circles,
        then the two circles are distinct.
        \item
        If $L$ is an isolated periodic orbit,
        then on each side of $L,$
        the flow is either topologically attracting towards $L$
        or topologically repelling away from $L.$
    \end{enumerate} \end{thm}
\begin{proof}
    Compactify the cylinder to a sphere $S^2$
    by adding a fixed point at each end.
    By slowing down the flow at each end,
    we can extend it to a flow defined on $S^2$
    with exactly two fixed points.
    We can then apply the Poincar\'e-Bendixson theorem
    for flows on $S^2.$
    All the above stated results can then be deduced from this
    and we leave most of the details to the reader.
    However, we give here a proof that an orbit cannot accumulate
    on a two-ended leaf, as this is slightly subtle.

    If an orbit accumulates on a point $p$
    which is not one of the two added fixed points,
    then this yields a trapping region given by a Jordan curve
    (as in the proof of the Poincar\'e-Bendixson theorem).
    This Jordan curve splits the sphere into two topological hemispheres
    with exactly one fixed point in each hemisphere.
    No $\omega$-limit set can contain both fixed points
    and so no $\omega$-limit set contains a two-ended leaf.
\end{proof}

\section{Foliations on a covering space} \label{sec:intercover} 

Suppose $\pi : M \to S$ defines a circle bundle over a
higher-genus surface $S.$
Let $\hat M$ denote the covering space of $M$ which is
obtained by pulling the circle bundle back by the covering $\bbH^2 \to S.$
In particular, $\hat M$ is homeomorphic to $\bbH^2 \ti S^1$
and we again use $\pi$ to denote the projection $\pi : \hat M \to \bbH^2.$

Consider a foliation on $M$ which satisfies the hypothesis
of \cref{thm:brittenham} and lift the foliation to the covering space $\Mhat.$
Since all of our analysis in this section occurs on $\Mhat,$
we use $\Fcal$ to denote the lifted foliation on $\Mhat$ and leave
the original foliation on $M$ unnamed here.
Let $h : \Mhat \to \Mhat$ be the lift
to $\Mhat$ of the homeomorphism given by \cref{thm:brittenham}.
In particular, $h$ is a finite distance
from the identity map on $\Mhat$ and for every leaf $L \in \Fcal,$
the image $h(L)$ is either transverse to the circle fibering
or is a union of fibers.
It therefore makes since to talk of the horizontal and vertical leaves
of the lifted foliation $\Fcal.$
Let $\Lam$ denote the sublamination of vertical leaves of $\Fcal.$

\begin{lemma} \label{lemma:horizontalplane}
    If $L$ is a horizontal leaf of $\Fcal,$ then
    $L$ is a topological plane and the map $\pi \circ h$
    when restricted to $L$ 
    is an embedding
    $\pi \circ h|_L : L \to \bbH^2.$
\end{lemma}
\begin{proof}
    Let $\Gcal = h(\Fcal)$ denote the foliation after $h$ has been applied.
    It is enough to prove for each horizontal leaf $L$ of $\Gcal$ that
    $L$ is homeomorphic to a plane and that $\pi|_L$ is an embedding.

    Suppose that $\hat \gam$ is an essential closed curve on the leaf $L.$
    The universal covering map $\tM \to \hat M$ is a cyclic covering
    which quotients a fibering of $\tM$ by lines down to
    a fibering of $\hat M$ by circles.
    Lift $\hat \gam$ to a curve $\tilde \gam$ inside a horizontal leaf
    $\tilde L \subof \tilde M.$
    As $\hat \gam$ is essential,
    the endpoints of $\tilde \gam$ are distinct and lie on the same fiber.
    This fiber is everywhere transverse to the foliation on $\tM$
    and so the segment between the two endpoints
    produces a transverse curve connecting two distinct points on
    the leaf $\tilde L.$
    For a lift of a taut foliation to the universal cover,
    such a transverse curve is not possible,
    giving a contradiction.

    Suppose now that $\pi|_L$ is not injective. Then there is a
    curve $\hat \gam$ inside of $L$ which starts and ends on the same fiber.
    The same steps as above again give a contradiction.
    Therefore $\pi|_L$ must be injective.
    As $L$ is transverse to the circle fibering, $\pi|_L$ is an open map.
\end{proof}
If the foliation contains both horizontal and vertical leaves,
then the embedding $\pi \circ h|_L$ given in \cref{lemma:horizontalplane}
will not have a uniformly continuous inverse
since the angle between the leaf and the fibers
tends to zero as we approach the vertical sublamination $\Lam.$
However, the map behaves nicely on subsets bounded
away from the vertical sublamination.
This behaviour is captured in the following result
which will be needed in \cref{sec:nooneended}.

\begin{lemma} \label{lemma:awayfromvert}
    Let $L$ be a horizontal leaf of $\Fcal$ and let $D \subof L$
    be such that $\dist(D, \Lam) > 0.$
    Then $\pi \circ h|_D$ has a uniformly continuous inverse.
\end{lemma}
\begin{proof}
    Again let $\Gcal = h(\Fcal)$ denote the foliation after $h$ has been applied.
    Let $\Gam = h(\Lam)$ be the vertical sublamination of $\Gcal.$
    Since $h \inv$ is uniformly continuous,
    it is enough to prove for
    a horizontal leaf $L$ of $\Gcal$ and a subset $D \subof L$
    with $\dist(D, \Gam) > 0$ that
    $\pi|_D$ has a uniformly continuous inverse.

    Suppose not.
    Then there are a constant $a > 0$ and
    sequences $\{x_n\}$ and $\{z_n\}$ in $D$
    such that
    \[    
        \lim_{n \to \infty} d_L(x_n, z_n) \ge a
        \qandq
        \lim_{n \to \infty} d(\pi(x_n), \pi(z_n)) = 0.
    \]
    where $d_L(x_n, z_n)$ is distance measured inside of the leaf $L.$
    We assume that $a < \dist(D, \Gam).$
    For each $n,$ let $\al_n$ be a short path in $\pi(L)$
    going from $\pi(x_n)$ to $\pi(z_n)$ and chosen so that
    the length of $\al_n$ tends to zero.
    Then we may choose points $\{y_n\}$ along $\pi \inv \al_n$ such that
    \[
        \lim_{n \to \infty} d_L(x_n, y_n) = a
        \qandq
        \lim_{n \to \infty} d(\pi(x_n), \pi(y_n)) = 0.
    \]
    By applying deck transformations and passing to subsequences,
    we may find convergent sequences $\{x'_n\}$ and $\{y'_n\}$ in $\Mhat$
    such that for all $n,$ the points
    $x'_n$ and $y'_n$ are on the same horizontal leaf $L_n$ and
    \[
        \dist(\pi(x'_n), \pi(\Gam))
        \ \ge \ \dist(\pi(D), \pi(\Gam)) \ > \ 0.
    \]
    Further,
    $\lim_{n \to \infty} d_{L_n}(x'_n, y'_n) = a,$ and
    $\lim_{n \to \infty} d(\pi(x'_n), \pi(y'_n)) = 0.$
    Then, $\{x'_n\}$ and $\{y'_n\}$ converge to distinct
    points $x$ and $y$ which are on the same circle fiber
    and on the same horizontal leaf,
    a contradiction.
\end{proof}
We will also use the following in order to apply \cref{lemma:awayfromvert}
to a region bounded by a center leaf.

\begin{lemma} \label{lemma:makedisk}
    Suppose $L$ is a horizontal leaf and
    $\Lc$ is a properly embedded curve in $L$ such that $\pi \circ h(\Lc)$
    is at finite distance from a geodesic ray in $\bbH^2$
    and $\dist(\Lc, \Lam) > 0$.
    Then there is a topological disk $D \subof L$
    such that $\del D = \Lc$ and $\dist(D, \Lam) > 0$.
\end{lemma}
\begin{proof}
    As in the previous proof, we work with $\Gcal = h(\Fcal)$
    and $\Gam = h(\Lam)$ and assume that $L$ is a leaf of $\Gcal.$
    Define $\al = \pi(\Lc).$
    This is a properly embedded curve which is at finite distance
    from a geodesic ray in $\bbH^2.$
    By temporarily viewing $\bbH^2$ as the unit disk in $\bbR^2$ and using
    the Jordan curve theorem, we can show that
    $\al$ bounds a topological disk $D_0 \subof \bbH^2$
    which is also at finite distance from the same geodesic ray.

    For every vertical leaf $L'$ in $\Gam,$
    the image $\pi(L')$ is a quasigeodesic curve
    which has two distinct endpoints on the circle at infinity $\del \bbH^2.$
    Such a curve cannot be a subset of $D_0,$
    and as $\al = \del D_0$ is disjoint from $\pi(\Gam),$
    it follows that $\pi(L')$ is disjoint from $D_0.$
    This holds for all leaves in $\Gam$ and
    implies that $D_0$ is a subset of $\pi(L).$
    Define $D \subof L$ as the preimage of $D_0$
    under the embedding $\pi|_L.$
    Then,
    \[
        \dist(\Lc, \Gam) > 0 \ \ \Rightarrow \ \
        \dist(\al, \pi(\Gam)) > 0 \ \ \Rightarrow \ \
        \dist(D_0, \pi(\Gam)) > 0 \ \ \Rightarrow \ \
        \dist(D, \Gam) > 0.
        \qedhere
    \] \end{proof}

\section{Proving \cref{prop:approxgood}} \label{sec:firstthree} 

We now focus on proving \cref{prop:approxgood}.
For this, we must show that if a partially hyperbolic diffeomorphism
$f : M \to M$
satifies the hypotheses of \cref{prop:approxgood}, then its approximating
foliations $\Fcsep$ and $\Fcuep$ satisfy all seven conditions listed in
\cref{thm:twofolns}.
In this section, we show conditions (1), (2), (3), which have 
short and straightforward proofs.
\Cref{sec:wellplaced} establishes quantitative estimates on the
contraction or expansion of vertical leaves and
shows conditions (4), (5), (6).
In \cref{sec:nooneended}, we show condition (7) which rules out one-ended leaves.

In this section, let $f : M \to M$ be a diffeomorphism as in \cref{prop:approxgood}.
Let $\Wcsbran$ and $\Wcubran$ be its branching foliations, and let
$\Fcsep$ and $\Fcuep$ be the approximating foliations as defined in \cref{sec:bran}.
We first prove conditions (1) and (2).

\begin{lemma} \label{lemma:onetwo}
    The foliations $\Fcsep$ and $\Fcuep$ are $C^0$ regular with $C^1$ leaves
    tangent to a continuous distribution and neither of the foliations has
    compact leaves.
\end{lemma}
\begin{proof}
    The regularity is part of the approximation assumption in \cref{thm:approxfoln}.
    The existence of a closed leaf of $\Fcsep$ or $\Fcuep$ implies the existence of a
    closed $f$-invariant torus in $\Wcsbran$ or $\Wcubran$ which is excluded by
    \cite{RHRHU-tori}.
\end{proof}
To prove condition (3) we use a standard \emph{volume vs.~length} argument.

\begin{lemma} \label{lemma:nocloseleaves}
    There are no pairs of vertical leaves $L, L'$ in $\Fcsep$
    (or in $\Fcuep$)
    which lift to leaves at bounded Hausdorff distance from each other. 
\end{lemma}
\begin{proof}
    Since the collapsing map $\hcsep$ moves leaves
    by a small distance, it is enough to prove this result for
    two leaves in the branching foliation $\Wcsbran.$
    Assume by contradiction that such a pair of vertical leaves exists.
    That is, $L$ and $L'$ in $\Wcsbran$ are vertical leaves with
    lifts to $\tM$ which are at finite Hausdorff distance.
    Denote by $\widetilde D$ the open region in between
    the lifts of $L$ and $L'$.
    This projects to a region $D$ in $M$ of bounded volume between
    $L$ and $L'$.
    We may freely assume that $L$ and $L'$ are as far apart as possible,
    so that the region $D$ is as large as possible.
    That is, assume that if $L''$ is any other vertical leaf with
    its lift at finite distance from the lift of $L,$
    then $L''$ lies in the closure of $D.$

    Let $J$ be a small unstable segment inside of $D.$
    Since the length of $f^n(J)$
    grows under forward iteration, there is 
    $\delta > 0$ such that
    $\volume(f^n(D)) > \delta$
    for all $n \ge 0.$
    As $M$ has finite volume, there is $n > 0$ such that
    $f^n(D)$ intersects $D.$ The assumption of $L$ and $L'$ as far apart as possible
    then implies that $f^n(D) = D.$

    On the universal cover, the stabilizer of the region $\widetilde D$
    under the deck transformations $\pi_1(M)$ is
    exactly the infinite cyclic subgroup associated to the circle fibering.
    Therefore, $D \subof M$ has fundamental group isomorphic to $\bbZ.$
    The iterates of $J$ stay at a bounded distance away from the boundary of $D.$
    If they did not, then there would be a short unstable segment going
    between
    an endpoint of $f^{n k}(J)$ and the boundary of $D,$
    and this unstable segment would shrink under iteration.
    We are therefore in the setting of 
    \cite[Proposition 5.2]{hps2018seif},
    which implies that the action of $f_*$ on $\pi_1(D)$
    has an eigenvalue larger than $1$.
    As $\pi_1(D)$ is infinite cyclic, this is impossible.
\end{proof}
%

\section{Intersections of vertical leaves} \label{sec:wellplaced} 

This section analyzes the dynamics on vertical
leaves of the branching foliations and their intersections
and establishes conditions (4), (5), (6) for \cref{prop:approxgood}.
As in \cref{sec:firstthree}, we assume throughout that $f : M \to M$
satisfies the hypotheses of \cref{prop:approxgood}.
Let $\Lamcs$ denote the branching sublamination
consisting of the vertical leaves of $\Wcsbran,$
and let $\Lamcu$ be the same for $\Wcubran.$

To state some of the results,
we define a notion of a ``well-placed'' center or stable segment.
Let $\Lcs$ be a leaf of $\Wcsbran$ and let $\Lc$ be a leaf of $\Wcbran$ inside of $\Lcs.$
A connected subset $\Jc \subof \Lc$ is \emph{well-placed} if
every stable leaf intersects $\Jc$ at most once.
If this property holds for $\Jc = \Lc,$ we that $\Lc$ is a well-placed leaf.
Similarly, if $\Ls$ is a stable leaf in $\Lcs,$ then
a connected subset $\Js \subof \Ls$ is \emph{well-placed} if
it intersects every leaf of the branching center foliation at most once.

Let $X$ be a subset of a leaf $L$ of either $\Wcsbran$ or $\Wcubran.$
Then for $r > 0,$ define a neighbourhood $U_r(X) \subof L$
by $y \in U_r(X)$ if and only if $\dist(y,X) < r$
where distance is measured inside $L.$

\begin{lemma} \label{lemma:csoneended}
    Let $\Lcs$ be a vertical leaf of $\Wcsbran$ and let
    $\Lc$ be a one-ended center leaf inside of $\Lcs.$
    Then
    \begin{enumerate}
        \item the set $\Lcs \sans \Lc$ has two connected components,
        one of which, call it $D,$
        is homeomorphic to an open disk;
        \item
        any segment $\Jc \subof \Lc$ is well-placed; and
        \item
        any stable segment $\Js$ inside of $D$
        is well-placed.
    \end{enumerate} \end{lemma}
\begin{proof}
    To prove item (1), compactify $\Lcs$ to a 2-sphere
    by adding a point at each end. As a one-ended leaf,
    $\Lc$ is properly embedded in $\Lcs$ and so its
    one-point compactification gives an embedding of
    a circle in the 2-sphere.
    Item (1) then follows from the Jordan curve theorem.

    For item (2), note that we can orient $\Es$ so that
    it points into $D$ everywhere along $\Lc.$ Therefore, a stable
    leaf cannot cross $\Lc$ twice as it would have to enter
    $D$ twice without leaving $D.$

    Item (3) follows from a standard Poincar\'e-Bendixson argument.
    Note here that as $\Wcbran$ is a branching foliation, a center leaf $L$
    though a point $x \in D$ might intersect the one-ended leaf $\Lc,$
    but $L$ cannot topologically cross $\Lc$ and so $L \subof \Lc \cup D.$
\end{proof}
\begin{lemma} \label{lemma:wellplaced}
    There exists a uniform constant $K_0 > 0$ such that if
    $J$ is a well-placed center or stable segment,
    then $\length(J) \le K_0 \area(U_1(J))$.
\end{lemma}
\begin{proof}
    This is an adaptation to the two-dimensional setting
    of the volume-versus-length argument of \cite{BI}.
    See \cite[Proposition 2.8]{hh2021surfaces} for details.
\end{proof}
Let $\hcsep : M \to M$ denote the collapsing map given by \cref{thm:approxfoln} and
let $\hcs : M \to M$ be the ``Brittenham'' homeomorphism given by
\cref{thm:brittenham} applied to $\Fcsep.$
We apologize for using similar names for the two functions,
but we were unable to find a better way to name them.

For each vertical leaf $L$ of $\Wcsbran,$ define a circle fibering on $L$ as follows.
Let $\Lep$ be the approximating leaf in $\Fcsep.$ Then $\hcs(\Lep)$ is a union of fibers
of $\pi$ and this pulls back to define a circle fibering on $\Lep.$
Apply $\hcsep|_{L_\ep}$ to define a circle fibering on $L$ itself.
For a point $x \in L,$ let $C_x$ denote the circle through $x.$
If $x,y \in L$ lie on distinct fibers, let $A_{x,y}$ denote the closed annulus
in $L$ which has $C_x$ and $C_y$ as its two boundary components.

Since both $\hcsep$ and $\hcs$ lift to maps on $\Mhat$ which are
a finite distance from the identity,
these circle fibers have uniformly bounded diameters. Up to rescaling the
metric on $M,$ we assume that $\diam(C_x) < 1$ for every circle fiber of every leaf
of $\Wcsbran.$

\begin{lemma} \label{lemma:annulus}
    There exists a uniform constant $K > 0$ such that
    if $L$ is a vertical leaf of $\Wcsbran,$
    $x$ and $y$ are points on $L$ with $d_L(x,y) > 2,$ and
    $J$ is a well-placed center or stable segment,
    contained in the annulus $A_{x,y},$
    then
    \[
        \length(J) \le K d_L(x,y).
    \] \end{lemma}

\begin{proof}
    For the vertical leaves of $\Wcsbran,$
    the area of a ball $B_r(x) \subof L$ is proportional to its radius $r.$
    That is, for large $r > 0,$
    the ratio $\area(B_r(x))/r$ is uniformly bounded
    independent of the leaf $L$ or the point $x \in L.$
    Further, the vertical leaves of $\Wcsbran$ are uniformly
    quasi-isometrically embedded in $\Mhat.$
    Therefore for an annulus $A_{x,y},$ the ratio
    $\area(U_1(A_{x,y}))/d_L(x,y)$ is also uniformly
    bounded so long as $x$ and $y$ are far enough apart.
    By \cref{lemma:wellplaced},
    \[
        \length(J) \le K_0 \area(U_1(J)) \le K_0 \area(U_1(A_{x,y})).
        \qedhere
    \] \end{proof}
Let $L \in \Lamcs$ be a periodic vertical leaf.
We say that $f$ is \emph{coarsely contracting} on $L$
if there is a compact set $X_0 \subof L$
such that for any compact subset $X \subof L$ there exists $n \ge 1$
such that $f^n(X) \subof X_0.$

\begin{lemma} \label{lemma:coarsecontract}
    For every periodic leaf $L \in  \Lamcs,$
    $f$ is coarsely contracting on $L.$
\end{lemma}
\begin{proof}
    Without loss of generality, we can replace $f$ by an
    iterate and assume $f(L) = L.$
    Note that $L$ is both quasi-isometric and $C^1$-diffeomorphic to
    $\bbR \ti S^1$ equipped with the standard metric.
    We therefore, for simplicity,
    prove the result for a diffeomorphism
    $f : \bbR \ti S^1 \to \bbR \ti S^1$
    which has a stable foliation and a center branching foliation
    and uniform constants $\sig < 1 < \eta$ such that
    $ \| Df v^s \| < \sig$
    for any unit vector in $\Es$
    and $ \| Df v \| < \eta$ for any unit vector in the tangent bundle of
    $\bbR \ti S^1$
    and for which the previous lemmas of this section hold.
    As all of the estimates are coarse estimates,
    the proof readily adapts to the original leaf $L.$

    By \cref{lemma:annulus}, there is $K > 0$ such that
    if $J$ is a well-placed stable segment inside of the annulus
    $[a,b] \ti S^1,$ then
    $\length(J) \le K (b-a)$.
    Since $K$ depends only on the stable foliation and the center branching
    foliation, we are free to replace $f$ by an iterate
    and assume that $\sig K < 1.$

    We will show by cases that if $x > 0$ is sufficiently large,
    then the fiber $\{x\} \ti S^1$ is mapped by some iterate $f^n$
    into $(-\infty, x-1] \ti S^1.$
    By symmetry, a similar result holds for $x < 0$ sufficiently large
    in absolute value and together these show coarse contraction.

    \smallskip

    \textbf{Case One:} assume for any interval $[a,b] \subof \bbR$
    that there is a well-placed stable segment with endpoints
    on $\{a\} \ti S^1$ and $\{b\} \ti S^1.$

    Let $r > 0$ be such that
    $f(\{0\} \ti S^1) \subof [-r,r] \ti S^1.$
    Let $x > 0$ be very large and consider a well-placed stable segment $J$
    with endpoints on $\{0\} \ti S^1$ and $\{x\} \ti S^1.$
    Then
    \[
        \length(f(J)) < \sig \length(J) < \sig K x
    \]
    and since one endpoint of $f(J)$ lies in $[-r,r] \ti S^1,$
    the other endpoint lies in $[-\sig K x -r, \sig K x + r] \ti S^1.$
    The set $f(\{x\} \ti S^1)$ has diameter at most $\eta$ and so
    \[
        f(\{x\} \ti S^1) \subof [-\sig K x -r - \eta , \sig K x + r + \eta] \ti S^1.
    \]
    If $x > 0$ is sufficiently large then $\sig K x + r + \eta < x - 1.$

    \smallskip

    \textbf{Case Two:}
    assume that there is a one-ended stable leaf tending to $+\infty.$

    By this, we mean a properly embedded leaf $\Ls$ such that
    $\pi(\Ls) \subof \bbR$ is of the form $[x_0, +\infty)$ for some $x_0 \in \bbR.$
    By \cref{lemma:csoneended}, $\Ls$ is well-placed.
    Let $r > 0$ be such that
    $f(\{x_0\} \ti S^1) \subof [-r,r] \ti S^1.$
    Let $x > 0$ be very large and consider a well-placed stable segment $J$
    with endpoints on $\{x_0\} \ti S^1$ and $\{x_0 + x\} \ti S^1.$
    Then adapting the arguments from Case One,
    \[
        f(\{x_0 + x\} \ti S^1)
        \subof [-\sig K x -r - \eta , \sig K x + r + \eta] \ti S^1.
    \]
    If $x > 0$ is sufficiently large then $\sig K x + r + \eta < x_0 + x - 1.$

    \smallskip

    \textbf{Case Three:}
    assume neither Case One or Two holds.

    If all stable leaves are well-placed, then we are in Case One,
    so consider a stable leaf which is not well-placed.
    Then there are a stable segment and a center segment
    with the same endpoints, from which we can construct
    a circle $\al$ transverse to the stable foliation.
    This circle cannot be nullhomotopic and so it is homotopic
    to a fiber, as is its image $f(\al).$
    The union of the bounded components of
    $(\bbR \ti S^1) \sans (\al \cup f(\al))$
    consists either of an annulus (if $\al$ and $f(\al)$ are disjoint)
    or a union of bigons, each of whose boundary is made of two arcs
    transverse to the stable direction.
    Using a Poincar\'e--Bendixson argument in each of these components,
    we see that every stable leaf through a point in $\al$
    must intersect a point in $f(\al)$ and vice versa.
    This implies that $\al$ and $f(\al)$ have the same stable saturate
    $\Ws(\al) = \Ws(f(\al)).$
    By a graph transform argument (see \cite[Lemma H.1]{bffp1}
    or \cite[Theorem 2.4]{ham2018prop}) there is an invariant circle
    $\bt = f(\bt)$ tangent to $\Ec.$

    As there are no stable circles, the boundary of the set $\Ws(\al) = \Ws(\bt)$
    is either empty or
    consists of one-ended stable leaves, all of which are well-placed.
    Since we are assuming that Case Two does not hold,
    none of these one-ended leaves tends to $+\infty$ and therefore
    there is $x_0 \in \bbR$ such that
    $[x_0, \infty) \ti S^1 \subof \Ws(\bt).$
    One can then show that $f(\Ws(\bt)) = \Ws(\bt).$
    For any fiber $\{x\} \ti S^1$ with $x > x_0,$
    we have that $f^n(\{x\} \ti S^1)$ converges to $\bt$
    and we can use this to prove coarse contraction.
\end{proof}
For the next lemma,
recall the definition of a center circle given in \cref{sec:bran}.

\begin{lemma} \label{lemma:circleperiodic}
    If a periodic leaf $L \in \Lamcs$ has a center circle,
    then it has a periodic center circle.
\end{lemma}
\begin{proof}
    Taking an iterate, assume $L$ is $f$-invariant and
    let $\al$ be a center circle.
    Similar to the proof of \cref{lemma:coarsecontract},
    we have that $\Ws(\al) = \Ws(f(\al))$
    and can find an invariant circle $\bt = f(\bt)$ tangent to $\Ec$
    and such that $\Ws(\al) = \Ws(\bt).$
    In this case, since $\{f^n(\al)\}$ is a sequence of
    center leaves which converge uniformly to $\bt,$
    it follows that $\bt$ is a leaf in the center branching foliation.
\end{proof}
\begin{lemma} \label{lemma:properint}
    For vertical leaves $\Lcs \in \Lamcs$ and $\Lcu \in \Lamcu,$
    any center leaf in the intersection $\Lcs \cap \Lcu$
    is properly embedded in both $\Lcs$ and $\Lcu.$
\end{lemma}
\begin{proof}
    Consider the approximating leaves $\Lcs_\ep$ and $\Lcu_\ep$
    and lift these to surfaces $\hat L^{cs}_\ep$ and $\hat L^{cu}_\ep$
    embedded in $\Mhat.$
    Since $\hat L^{cs}_\ep$ and $\hat L^{cu}_\ep$ are transverse to each other
    and each is properly embedded in $\Mhat,$
    the intersection of the two surfaces is also
    properly embedded in $\Mhat.$
    It follows that each connected component $\Lc_\ep$
    of $\Lcs_\ep \cap \Lcu_\ep$ is properly embedded inside of $\Lcs_\ep$
    and $\Lcu_\ep.$
    From \cref{prop:centerapprox}, the result follows.
\end{proof}
\begin{lemma} \label{lemma:periodcircles}
    If $\Lcs$ and $\Lcu$ are periodic vertical leaves,
    then every center leaf in the intersection $\Lcs \cap \Lcu$ 
    is a circle.
\end{lemma}
\begin{proof}
    Suppose the intersection has a non-compact center leaf $\Lc.$
    As in the last proof, we will need to lift the leaves
    to $\Mhat.$ Without loss of generality, we can assume
    $f(\Lcs) = \Lcs$ and $f(\Lcu) = \Lcu.$
    Choose a lifted map $\fhat : \Mhat \to \Mhat$
    and a lifted leaf $\Lcshat$ such that $\fhat(\Lcshat) = \Lcshat.$
    Choose lifts of $\Lc$ and $\Lcu$ such that the intersection
    $\Lcshat \cap \Lcuhat$ contains the non-compact leaf $\Lchat.$
    Then there is a deck transformation $\Psi : \Mhat \to \Mhat$
    of the covering map $\Mhat \to M$ such that
    $\fhat(\Lcuhat) = \Psi(\Lcuhat).$
    The intersection $\Lcshat \cap \Psi(\Lcuhat)$ is non-compact
    since it contains $\fhat(\Lchat).$
    We wish to show that $\Psi$ is the identity.

    Suppose not. Then there is a non-trivial deck transformation
    $\psi : \bbH^2 \to \bbH^2$ such that $\psi \circ \pi = \pi \circ \Psi$
    where $\pi$ denotes both 
    the map $M \to S$ and the map $\Mhat \to \bbH^2$ which define the circle fiberings
    on $M$ and $\Mhat$ respectively.
    By \cref{thm:ideal}, $\pi(\Lcuhat)$ shadows a geodesic $\ell$ in $\bbH^2.$
    Non-compactness of $\Lcshat \cap \Psi(\Lcuhat)$
    implies that the geodesics $\psi(\ell)$ and $\ell$
    share a point on the circle at infinity.
    This means that when we quotient $\bbH^2$ down to the closed surface $S,$
    the image of $\ell$ is either a closed geodesic or
    it spirals into a closed geodesic.
    Either way, the geodesic lamination given by \cref{thm:ideal}
    would contain a compact leaf.
    This would imply that $\Lamcu$ contains a compact leaf, a contradiction.
    This shows that $\Psi$ is the identity and so
    the lifts satisfy both $\fhat(\Lcshat) = \Lcshat$
    and $\fhat(\Lcuhat) = \Lcuhat.$

    By \cref{lemma:properint},
    there is a sequence $\{p_k\}$ in $\Lchat$
    which escapes every compact subset of $\Mhat.$
    By coarse contraction of $f$ on $\Lcs,$
    there is sequence of integers $\{n_k\}$ tending to $+\infty$
    such that $\{q_k\}$ defined by $q_k = \fhat^{n_k}(p_k)$
    is a bounded sequence inside of $\Lcshat.$
    Then $\{q_k\}$ lies inside of a compact set inside
    of $\Lcshat,$ but the sequence $\{p_k\} = \{\fhat^{-n_k}(q_k)\}$
    is unbounded. This contradicts the coarse contraction
    of $f \inv$ on $\Lcu.$
\end{proof}

We now combine the above properties
with results from \cite{hps2018seif} to analyse
the intersections of the vertical laminations.
We say that a leaf $L$ in $\Lamcs$ or $\Lamcu$
is an \emph{accessible boundary leaf} if
the approximating leaf $L_\ep$ in $\Lamcs_\ep$ or $\Lamcu_\ep$
lies on the accessible boundary of that lamination.

\begin{lemma} \label{lemma:accboundary}
    The accessible boundary leaves are dense in $\Lamcs$
    and in $\Lamcu$ and every accessible boundary leaf is periodic.
\end{lemma}
\begin{proof}
    Density of these leaves follows from
    \cref{thm:ideal} and the fact that a geodesic lamination on a surface
    has empty interior \cite[\S 1.7.4]{calegari2007book}.
    Periodicity is given by \cite[Lemma 3.5]{hps2018seif}.
\end{proof}
\begin{lemma} \label{lemma:twocircles}
    Let $L$ be an accessible boundary leaf of $\Lamcs$ or $\Lamcu$
    and suppose that there is a periodic center circle
    $\gam = f^k(\gam) \subof L.$
    Then any center circle in $L$ must intersect $\gam.$
\end{lemma}
\begin{proof}
    Suppose $\al$ is a circle in $L$ disjoint from $\gam.$
    This supposition is the same as the conclusion
    of \cite[Lemma 4.4]{hps2018seif}.
    We now follow all of the steps in that paper
    from immediately after that lemma until
    the end of \cite[\S 4]{hps2018seif}
    to arrive at a contradiction.
    Those steps do not rely on $f$ being dynamically coherent
    and apply equally well to the branching foliations.
    The adapted proof holds for an accessible boundary leaf $L$
    of an $f$-invariant branching sublamination $\Lam \subof \Lamcs.$
    The lamination does not need to be minimal and we use
    $\Lam = \Lamcs$ here.
\end{proof}

\begin{lemma} \label{lemma:accintersect}
    If accessible boundary leaves
    $\Lcs \in \Lamcs$ and $\Lcu \in \Lamcu$ intersect,
    then their intersection is a single circle.
\end{lemma}
\begin{proof}
    Without loss of generality, assume that
    $f(\Lcs) = \Lcs$ and $f(\Lcu) = \Lcu.$
    By \cref{lemma:periodcircles}, the intersection is a union of center circles.
    By \cref{lemma:circleperiodic},
    $\Lcu$ contains a periodic center circle $\bt.$
    By \cref{lemma:twocircles}, all of the circles in $\Lcs \cap \Lcu$ must intersect $\bt.$
    Suppose there is a circle $\al$ distinct from $\bt.$
    Then a segment of $\bt$ and a segment of $\al$ together
    bound a disk $D$ inside of $\Lcu.$
    Consider an unstable segment $\Ju \subof D.$
    By iterating forward and using a length versus area argument,
    we see that the diameter of $f^n(D)$ tends to infinity
    when regarded as a subset of $\Lcu.$
    This means that the intersection $\Lcs \cap \Lcu$
    must contain arbitrarily long center segments $\Jc_n$ with one endpoint
    on $\bt.$ Taking $n \to \infty,$
    these center segments accumulate on
    an unbounded center ray inside of $\Lcs \cap \Lcu,$
    contradicting \cref{lemma:periodcircles}.
    Therefore the intersection $\Lcs \cap \Lcu$ must consist of the single
    circle $\bt.$
\end{proof}
\begin{lemma} \label{lemma:commoncircle}
    Distinct accessible boundary leaves $L_1, L_2 \in \Lamcs$
    cannot intersect the same accessible boundary leaf $\Lcu \in \Lamcu.$
\end{lemma}
\begin{proof}
    Suppose such leaves exist and let
    $\al_i$ denote the intersection $L_i \cap \Lcu.$
    Inside the leaf $\Lcu,$
    both circles are transverse to the unstable direction
    and as in the proof of \cref{lemma:coarsecontract},
    it follows that
    $\Wu(\al_1) = \Wu(\al_2).$
    As each circle is periodic, it further holds that $\al_1 = \al_2.$
    Since these are branching foliations instead of true folations,
    this does not immediately imply that $L_1 = L_2$
    and more work is needed to produce a contradiction.
    
    By \cref{lemma:nocloseleaves}, the lifts of $L_1$ and $L_2$ to $\Mhat$
    cannot be at finite distance from each other and so
    they shadow two distinct geodesics $\ell_1$ and $\ell_2$ on $\bbH^2.$
    We first consider the case where $\ell_1$ and $\ell_2$ have
    four distinct endpoints on $\del \bbH^2.$
    This means that the intersection $L_1 \cap L_2$ is compact.
    Since the unstable bundle $\Es$ is uniquely integrable,
    this intersection includes the basin of attraction
    $\Ws(\al_1) = \Ws(\al_2)$ of the center circles 
    which is therefore precompact in each of $L_1$ and $L_2.$
    Its boundary would be compact and stable saturated,
    which would imply a stable circle, a contradiction.

    In the slightly more complicated case where $\ell_1$ and $\ell_2$
    yield three distinct points on $\del \bbH^2,$
    we first split $L_1$ and $L_2$ into half-leaves 
    by cutting along $\al_1$ and then do the above steps
    on those half leaves which tend to distinct points on $\del \bbH^2.$
\end{proof}
\begin{lemma} \label{lemma:acclamintersect}
    For an accessible boundary leaf $\Lcu \in \Lamcu,$
    the intersection of $\Lamcs$ with $\Lcu$ is either empty
    or consists of a single circle.
\end{lemma}
\begin{proof}
    This follows from lemmas
    \ref{lemma:accintersect} and \ref{lemma:commoncircle}
    and the fact that the accessible boundary leaves are dense in $\Lamcs.$
\end{proof}
\begin{lemma} \label{lemma:goodvert}
    If a leaf $\Lcs \in \Lamcs$ intersects a leaf $\Lcu \in \Lamcu,$
    then both leaves are isolated and the intersection
    is a single circle.
\end{lemma}
\begin{proof}
    Suppose at least one of $\Lcs$ or $\Lcu$ is not
    isolated. Say $\Lcu$ is not isolated.
    Then there is a sequence of distinct accessible boundary leaves
    $\Lcu_n \in \Lamcu$ converging to $\Lcu$ and an accessible boundary leaf
    $\Lcs_1 \in \Lamcs$ close to or equal to $\Lcs.$
    By the transversality of the $\Ecu$ and $\Ecs$ subbundles,
    there is $\delta > 0$ such that if 
    $x \in \Lcs_1$ and $y \in \Lcu_n$ and $d(x,y) < \delta,$
    then the intersection $\Lcs_1 \cap \Lcu_n$ is non-empty.
    We may assume that the leaves were chosen so that
    such points exist for all $n.$
    This contradicts the previous lemma.
\end{proof}
\begin{lemma}
    There are finitely many vertical center circles.
\end{lemma}
\begin{proof}
    If not, there are points $p_n \in \Lamcs \cap \Lamcu$ on distinct center circles
    converging to a point $p \in \Lamcs \cap \Lamcu$
    and the previous lemma gives a contradiction.
\end{proof}
The last three lemmas show that the branching foliations
$\Wcsbran$ and $\Wcubran$ satisfy conditions (4), (5), (6)
listed in \cref{thm:twofolns}.
We now show that the
approximating foliations $\Fcsep$ and $\Fcuep$ also satisfy these conditions.

Suppose two vertical leaves $\Lcsep \in \Fcsep$ and $\Lcuep \in \Fcuep$
intersect and consider a connected component $\Lcep$ of the intersection.
Then as $\Lcsep$ and $\Lcuep$ each collapse to vertical leaves
of the branching foliation, $\hcep(\Lcep)$ is a circle.
As $\hcep$ is close to the identity, $\Lcep$ is bounded inside each
of $\Lcsep$ and $\Lcsep.$
By \cref{lemma:properint}, it is properly embedded as well and so
must be a circle.

Suppose two distinct leaves
$C_{11} \subof L^{cs}_{\ep,1} \cap L^{cu}_{\ep,1}$
and
$C_{22} \subof L^{cs}_{\ep,2} \cap L^{cu}_{\ep,2}$
of $\Fcep$ are both circles and project to the same
subset $\hcep(C_{11}) = \hcep(C_{22})$ of $M.$
Then $C_{11}$ and $C_{22}$ are $\ep$-close in Hausdorff distance
and the uniform transversality of $\Fcsep$ and $\Fcuep$
implies that that intersection
$L^{cs}_{\ep,1} \cap L^{cu}_{\ep,2}$
has a connected component $C_{12}$ close to both $C_{11}$ and $C_{22}.$
Since
\[
    \hcsep(L^{cs}_{\ep,1}) = \hcsep(L^{cs}_{\ep,2})
    \qandq
    \hcuep(L^{cu}_{\ep,1}) = \hcuep(L^{cu}_{\ep,2}),
\]
it follows that $\hcep(C_{12}) = \hcep(C_{11})$
and this contradicts \cref{lemma:commoncircle}.

It therefore holds that there is a bijective correspondence
between circles in $\Wcbran$ and circles in $\Fcep$
and that two circles of $\Wcbran$ lie in the same leaf of $\Wcsbran$ or $\Wcubran$
if and only if the corresponding circles of $\Fcep$
lie in the same leaf of $\Fcsep$ or $\Fcuep$ respectively.
From this, we see that conditions (4), (5), (6)
of \cref{thm:twofolns} holds for the approximating foliations.

\section{No one-ended center leaves} \label{sec:nooneended} 

This section concludes the proof of
\cref{prop:approxgood} by showing that the approximating foliations
satisfy condition (\ref{item:nooneended}) in the hypotheses of \cref{thm:twofolns}.
Any one-ended leaf for the approximating foliations will map
down to a one-ended center leaf for the original branching foliations.
Therefore, it is enough to show the following.

\begin{prop} \label{prop:nooneended}
    Let $f : M \to M$ be a partially hyperbolic diffeomorphism
    on a circle bundle over a higher-genus surface
    and suppose that $Df$ preserves the orientations
    of $\Eu, \Ec,$ and $\Es.$
    Then a vertical leaf of the cs branching foliation
    does not contain any one-ended center leaves.
\end{prop}
We first state a result on $\bbH^2.$
For a geodesic $\ell$ in $\bbH^2,$
define a projection $\pi_\ell : \bbH^2 \to \ell$
along geodesics perpendicular to $\ell.$
Note that a set $X \subof \bbH^2$ is at finite distance from $\ell$
if and only if $\dist(x, \pi_\ell(x))$ is uniformly bounded for all $x \in X.$
Consider a properly embedded topological line $\al$ in $\bbH^2$
at finite distance from $\ell.$
Similar to a curve in a cylinder,
we say $\al$ is \emph{one-ended} if it does not
intersect every fiber of $\pi_\ell.$
One may show the following properties for such a curve.

\begin{lemma} \label{lemma:hypdisk}
    Let $\al$ be a one-ended properly embedded line
    at finite distance from a geodesic $\ell$ in $\bbH^2.$
    Then $\al$ bounds an open topological disc $D_0 \subof \bbH^2$
    with the following properties.
    \begin{enumerate}
        \item for every $z \in \pi_\ell(\al),$
        the fiber $\pi_\ell \inv (z)$ intersects $D_0$
        in a set of uniformly bounded diameter;
        \item
        the curve $\al$ may be split at a point $x_0$
        into two closed rays $\al_+$ and $\al_-$
        such that
        \[
            \al_+ \cup \al_- = \al,
            \quad
            \al_+ \cap \al_- = \{x_0\},
            \qandq
            \pi_\ell(\al_+) = \pi_\ell(\al_-) = \pi_\ell(\al);  \]
        \item
        if $\bt$ is a properly embedded topological ray in $D_0 \cup \al$
        starting at $x_0,$
        then $\pi_\ell(\bt) = \pi_\ell(\al).$
    \end{enumerate} \end{lemma}

\begin{figure}
    \centering
    \includegraphics{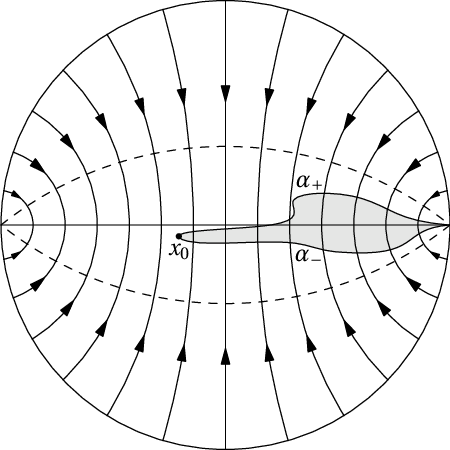}
    \caption{
    A graphical depiction of \cref{lemma:hypdisk}.
    The horizontal line here is the geodesic $\ell$
    in the Poincar\'e disk model of $\bbH^2$ and
    the two dashed curves are at a fixed distance from $\ell.$
    The arrows depict the projection $\pi_\ell$
    along geodesics perpendicular to $\ell.$
    The shaded region is the topological disk $D_0$
    which has boundary
    $\al_+ \cup \al_- = \al.$
    }
    \label{fig:banana}
\end{figure}
The details of the proof are left to the reader.
See \cref{fig:banana}.
As in the proof of \cref{lemma:makedisk},
the existence of the disk
$D_0$ follows from the Jordan curve theorem.

\medskip{}

We now prove \cref{prop:nooneended} by contradiction.

For the remainder of this section,
we work only in the covering space
$\Mhat.$
Therefore in this section,
$f : \Mhat \to \Mhat$ will denote the lift to $\Mhat$ of the diffeomorphism
given in the hypothesis of \cref{prop:nooneended}.
Similarly,
let $\Wcsbran$ and $\Wcubran$ denote the lift to $\Mhat$ of the branching foliations
tangent to $\Ecs$ and $\Ecu$ respectively.
Let $\ep > 0$ be small and let $\Fcsep, \Fcuep$ be lifts to $\Mhat$
of the approximating foliations given by \cref{thm:approxfoln}.
Similarly, let $\hcuep, \hcsep : \Mhat \to \Mhat$ be lifts to
$\Mhat$ of the corresponding collapsing maps.
Note that each of $\hcuep$ and $\hcsep$ is $\ep$-close to the identity map on $\Mhat.$

The definitions of horizontal, vertical, and one-ended leaves
work equally well in the setting of the branched foliations
$\Wcsbran$, $\Wcubran$, and $\Wcbran$
as for true foliations, and so we use them here.

In this section,
let $\pi : \Mhat \to \bbH^2$ be the projection which defines
the lifted circle fibering on $\Mhat.$
Since $\Fcsep$ is the lift of a true foliation
without compact leaves on $M,$
\cref{thm:brittenham} applies.
Let $\hcs : \Mhat \to \Mhat$ be the lift
to $\Mhat$ of the homeomorphism given by \cref{thm:brittenham}.
For the foliation $\Fcuep,$ define $\hcu : \Mhat \to \Mhat$ analogously.
As explained in \cref{sec:bran}, the branching foliations $\Wcsbran$ and $\Wcubran$
intersect in a branching center foliation $\Wcbran.$

Consider a vertical leaf $\Lcs$ of $\Wcsbran$ and assume that there
is a one-ended center leaf $\Lc$ inside of $\Lcs.$
Let $\Lcu$ be the leaf of $\Wcubran$ containing $\Lc.$
By \cref{lemma:goodvert}, $\Lcu$ must be a horizontal leaf.

Let $\Lcuep$ be the leaf in $\Fcuep$ which approximates $\Lcu;$
that is, $\hcuep(\Lcuep) = \Lcu.$
Define a curve $\Lcep \subof \Lcuep$ by
$\hcuep(\Lcep) = \Lc.$
Let $\Lamcu_{\ep}$ denote the sublamination of all vertical leaves of $\Fcuep.$

\begin{lemma} \label{lemma:Lcdist}
    The curve $\Lcep$ lies at positive distance
    from $\Lamcu_{\ep}.$
\end{lemma}
\begin{proof}
    Suppose not. Then there are sequences
    $\{x_n\}$ in $\Lcep$
    and $\{y_n\}$ in $\Lamcu_{\ep}$ such
    that $d(x_n, y_n)$ tends to zero.
    Since $\Lcep$ is disjoint from $\Lamcu_{\ep},$
    it must be the case that $\{x_n\}$ escapes every compact subset of
    $\Lcep.$
    Note that $\Lc$ is properly embedded in $\Lcs$
    which itself is properly embedded in $\Mhat,$
    and so $\Lcep,$ which is at finite distance from $\Lc,$
    is properly embedded in $\Mhat.$
    Therefore, $\{x_n\}$ escapes every compact subset of $\Mhat.$

    For large $n,$ the point $x_n$ is close to the vertical leaf $\Lcs$ of $\Wcsbran$
    and is close to some vertical leaf $L_n$ of $\Wcubran.$
    Therefore, $\Lcs$ and $L_n$ intersect in a center circle near $x_n.$
    This shows that $\Lcs$ contains infinitely many center circles,
    which contradicts \cref{lemma:goodvert}.
\end{proof}
Let $\ell$ be the unique geodesic in $\bbH^2$ such that
the vertical leaf $\Lcs$ lies at finite distance from $\pi \inv(\ell).$
In the statement of the next lemma,
we use the notation $[a,b]$ for a compact geodesic segment
inside of $\ell.$

\begin{lemma} \label{lemma:cudisk}
    Inside the horizontal leaf $\Lcu,$
    the curve $\Lc$ bounds an open topological disk $D$
    and there is a continuous function $q : \Lcu \to \ell$
    with the following properties:
    \begin{enumerate}
        \item there is a uniform constant $R_1 > 0$ such that
        if $x \in D,$
        $y \in \Mhat,$ and
        $q(x) = \pi(y),$
        then $d(x, y) < R_1$ as measured in $\Mhat,$ and
        \item
        there is a uniform constant $K_1 > 0$ such that
        if $[a,b] \subof q(D)$ with \\
        $\length([a,b]) \ge 1$, then
        there is an unstable segment $\Ju \subof D$
        such that $q(\Ju) = [a,b],$
        the endpoints $x^u$ and $y^u$ of $\Ju$ satisfy
        $q(x^u) = a,$
        $q(y^u) = b,$ and
        \[
            \length(\Ju) \le K_1 \length([a,b]).
        \] \end{enumerate} \end{lemma}
\begin{proof}
    By \cref{lemma:horizontalplane},
    $\pi \circ \hcu$ restricted to $\Lcuep$ is an embedding into $\bbH^2.$
    Therefore $\pi \circ \hcu(\Lcep)$ is a properly embedded curve in $\bbH^2$
    at finite distance from $\ell$ and which satisfies the hypotheses of
    \cref{lemma:hypdisk}.
    Let $D_0$ be the disk given there and define a map
    $q_\ep : \Lcuep \to \ell$ by
    $q_\ep = \pi_\ell \circ \pi \circ \hcu.$
    By lemmas \ref{lemma:makedisk} and \ref{lemma:Lcdist},
    there is a topological disk $D_\ep$ in $\Lcuep$ bounded by $\Lcep$
    such that
    $D_\ep$ is at positive distance from $\Lamcu_\ep$
    and $D_0 = \pi \circ \hcu (D_\ep).$

    Recall that the restriction of $\hcuep$ to $\Lcuep$
    is a $C^1$ diffeomorphism onto $\Lcu$ which
    is $\ep$-close to the identity in the $C^1$ topology.
    Let $g_\ep : \Lcu \to \Lcuep$ denote its inverse
    and define $q : \Lcu \to \ell$ as the composition
    \[
        q = q_\ep \circ g_\ep = 
        \pi_\ell \circ \pi \circ \hcu \circ g_\ep.
    \]
    In words, we apply $g_\ep$ to take the leaf $\Lcu$ of the branching foliation
    to its approximating leaf $\Lcuep,$ then apply the Brittenham map $\hcu$
    to isotope $\Lcuep$ to a position transverse to the circle fibers,
    then apply $\pi$ to project each point of the isotoped leaf down to $\bbH^2,$
    and finally apply $\pi_\ell$ to map each of these points in $\bbH^2$ to the closest
    point on the geodesic $\ell.$
    We also define $D = \hcuep(D_\ep).$

    \smallskip{}

    To prove item (1),
    consider $x \in D$ and $y \in \Mhat$ such that $q(x) = \pi(y).$
    Write $z = \pi \circ \hcu \circ g_\ep(x)$ so that $\pi_\ell(z) = q(x).$
    Since $\hcuep$ is $\ep$-close to the identity map on $\Mhat$
    and $\hcu$ is the lift to $\Mhat$ of a map isotopic to the identity on $M,$
    it follows that $d(\hcu \circ g_\ep(x), x)$ is uniformly bounded
    and therefore $d(z, \pi(x))$ is uniformly bounded.
    As $z \in D_0,$ it lies at a uniformly bounded distance from the geodesic $\ell$
    and so $d(q(x), \pi(x)) = d(\pi_\ell(z), \pi(x))$ is uniformly bounded.
    All together this implies that
    the distance $d(\pi(x), \pi(y))$ in $\bbH^2$ is uniformly bounded
    and so the distance $d(x, y)$ in $\Mhat$ is uniformly bounded as well.

    \smallskip{}

    To prove item (2),
    we adapt details of the ``area-versus-volume'' argument.
    First consider an integer $N > 1$ and an unstable segment
    $\Ju$ inside the disk $D$ such that
    $N-1 \le \length(\Ju) \le N$.
    Then there are points
    $x_1, \ldots, x_N \in \Ju$
    such that $d_u(x_i, x_j) \ge 1$
    for all distinct $i$ and $j$ and where $d_u$ is distance
    measured along $\Ju.$
    As $\Ju$ is well-placed, there is a uniform constant $r_1 > 0$
    such that $d_{cu}(x_i, x_j) > r_1$
    where $d_{cu}$ is distance measured in $\Lcu.$
    (For details,
    see \cite[Lemma 2.7]{hh2021surfaces}.)

    Write $z_i = \pi \circ \hcu \circ g_\ep(x_i)$
    and note that $z_0 \in D_0$ and $\pi_\ell(z_i) = q(x_i).$
    The composition $\pi \circ \hcu \circ \hcuep$
    restricts to an invertible map from $D$ to $D_0.$
    By \cref{lemma:awayfromvert}, it has a uniformly continuous inverse.
    This implies that there is $r_2 > 0$ independent of both $N$ and $\Ju$
    such that
    $d(z_i, z_j) > r_2$
    for distinct $i$ and $j$ and where distance is measured on $\bbH^2.$
    Then the disks of diameter $r_2$ centered at the points
    $z_1, \ldots, z_N$
    are pairwise disjoint and their union is a subset of $\bbH^2$
    with area proportional to $N.$

    Write $[a,b] = q(\Ju).$
    As $D_0$ lies at finite distance from the geodesic $\ell,$
    the set
    \[
        \{ z \in D_0 : \pi_\ell(z) \in [a,b] \} \subof \bbH^2
    \]
    has area proportional to the length of $[a,b]$.
    Moreover, the $r_2$-neighbourhood of this set
    also has area proportional to the length of $[a,b]$
    and contains the union of the disks described above.
    It follows that there is a uniform constant $K_1 > 0$ such that
    \[
        \length(\Ju) \le N \le K_1 \length(q(\Ju)).
    \]

    To conclude the proof, we also need to show
    for any compact geodesic segment $[a,b] \subof q(D)$
    that there exists an unstable segment $\Ju$ as stated in the lemma.
    Using the last item of \cref{lemma:hypdisk},
    we can find a point $x_0 \in \Lc$ such that
    if $\Lu$ is the unstable leaf through $x_0$ then
    $q(\Lu \cap D) = q(D).$
    For a geodesic segment
    $[a,b] \subof q(D),$
    define a subset $U \subof \Lu \cap D$
    by $x \in U$ if $q(x)$ lies in the interior of $[a,b]$.
    Then at least one connected component of $U$ has
    a closure $\Ju$ for which $q(\Ju) = [a,b].$
\end{proof}

In what follows, we use $D$ and $q$
as given by the above lemma.

\begin{lemma} \label{lemma:closepoints}
    There exists a uniform constant $R > 0$ such that
    if $x^u \in D,$
    $x^s \in \Lcs,$ and
    $q(x^u) = \pi_\ell \circ \pi(x^s),$ then
    $d(x^u, x^s) < R$
    where distance is measured in $\Mhat.$
\end{lemma}
\begin{proof}
    This follows from the previous lemma
    and the fact that $\pi(\Lcs)$ is at finite distance from $\ell.$
\end{proof}

We now adapt the arguments of \cite[\S 5.4]{hps2018seif}.
The rough idea is to use the one-ended center leaf
to find two long curves tangent to $\Es$ and $\Eu$
which are close to each other in Hausdorff distance.
Applying an iterate of the dynamics means that one curve
grows and the other shrinks, producing a contradiction.
The proof is similar to those of 
lemmas \ref{lemma:coarsecontract} and \ref{lemma:periodcircles},
but the inequalities are slightly more complicated.

Partial hyperbolicity implies the existence of
constants $\lam < \sig < \mu$ such that
\[
    0 < \lam < \| Df v^s \| < \sig < 1 < \mu < \| Df v^u \|
\]
for all unit vectors $v^s \in \Es$ and $v^u \in \Eu.$
Let $C > 0$ be a very big constant and $n$ a large positive integer
which will be determined by inequalities given at the end of the proof.

Consider the one-ended center leaf $f \invn(\Lc)$ inside the vertical
cs leaf $f \invn(\Lcs).$
By \cref{lemma:csoneended}, $f \invn(\Lc)$ bounds a disk $D^{cs}$ inside of $f \invn(\Lcs).$
Consider a stable leaf $\Ls$ which intersects $f \invn(\Lc)$
and let
$\Ls_+ = \Ls \cap D^{cs}$ denote the ray which is contained in $D^{cs}.$
As explained in the previous section,
we may assume that there is a circle fibering
defined on $f \invn(\Lcs).$
Take points $x,y \in f \invn(\Lcs)$ with $d_{cs}(x,y) = C$ as measured in $f \invn(\Lcs)$
and such that the circles fibers $C_x$ and $C_y$ both intersect $\Ls_+.$
Recall that $A_{x,y} \subof f \invn(\Lcs)$ denotes
the annulus between these two fibers.
Consider the connected components of the intersection
$\Ls_+ \cap A_{x,y}.$
One of these components is a stable segment $\Js$ such that
its endpoints $x^s$ and $y^s$ satisfy
$x^s \in C_x$ and $y^s \in C_y.$
By \cref{thm:ideal}, the vertical cs leaves are uniformly quasi-isometrically
embedded in $\Mhat$ and so there is a uniform constant $c_1 > 0$
such that $d(x^s, y^s) > c_1 C$ as measured in $\Mhat.$

By \cref{lemma:csoneended}, $\Js$ is well-placed and so by \cref{lemma:annulus},
$\length(\Js) \le K C.$
Let $x^s_n = f^n(x^s)$ and $y^s_n = f^n(y^s)$ denote the endpoints of $f^n(\Js).$
Then,
\[
    d(x^s_n, y^s_n) \le \sig^n \length(\Js) \le K \sig^n C.
\]
Let $a = \pi_\ell \pi(x^s_n)$ and $b = \pi_\ell \pi(y^s_n)$
and consider the geodesic segment $[a,b] \subof \ell.$
Since the map $\pi : \Mhat \to \bbH^2$ is a quasi-isometry,
there is a uniform constant $c_2 > 0$ such that
\[
    \length([a,b]) \ \le \ c_2 d(x^s_n, y^s_n) \ \le \ c_2 K \sig^n C.
\]
We may assume that $\Js$ was chosen far enough into one of the ends
of $f \invn(\Lcs)$ that $[a,b]$ is a subset of $q(D).$
Let $\Ju$ be the unstable segment given
by \cref{lemma:cudisk} such that
$q(\Ju) = [a,b]$ and
\[
    \length(\Ju) \ \le \ K_1 \length([a,b]) \ \le \ c_2 K_1 K \sig^n C.
\]
Let $x^u$ and $y^u$ denote the endpoints of $\Ju.$
By \cref{lemma:closepoints},
the distances
$d(x^s_n, x^u)$ and $d(y^s_n, y^u)$ are each less than $R.$
The constant $\lam$ is defined so that $f \inv$ expands
distance by at most $\lam \inv.$
Therefore, the distances
$d( x^s, f \invn(x_n) )$
and
$d( y^s, f \invn(y_n) )$ are each less than $R \lam \invn.$
As $\Ju$ is an unstable curve,
\[
    \length(f \invn(\Ju)) \ \le \ c_2 K_1 K ( \tfrac{\sig}{\mu} )^n C.
\]
By the triangle inequality for the metric on $\Mhat,$
\[
    c_1 C \ \le \ d(x^s, y^s) \ \le \ 
    c_2 K_1 K ( \tfrac{\sig}{\mu} )^n C \ + \ 2 R \lam \invn.
\]
Since $\sig < 1 < \mu,$ we can choose an integer
$n > 0$ large enough that $c_2 K_1 K (\tfrac{\sig}{\mu})^n < c_1.$
We can then choose $C$ large enough to produce a contradiction.
This concludes the proof of \cref{prop:nooneended}
which, combined with the results of the last two sections,
proves \cref{prop:approxgood}.

\smallskip

With \cref{prop:approxgood} proved, we no longer need to consider
partially hyperbolic dynamical systems. The remainder of the paper
is dedicated to proving \cref{thm:main} and does not involve
partial hyperbolicity.

\section{The averaged vector field} \label{sec:average} 

This section gives a proof of \cref{thm:twofolns} in the case where there
are no vertical circles in $\Fc.$ 
Throughout this section, assume that $\Fcs, \Fcu,$
and their intersection $\Fc$
are foliations as in \cref{thm:twofolns} and that
$\phic$ is the continuous flow on $M$ associated to $\Fc.$
We refer to $\Fc$ as the ``center'' foliation and $\phic$ as the ``center'' flow,
even though the foliations in \cref{thm:twofolns} do not necessarily need
to be associated to a partially hyperbolic system.
For the remainder of the paper, a ``center circle'' is a circular leaf in $\Fc.$
All of $\Fcs,$ $\Fcu,$ and $\Fc$ are true foliations and we no longer need to worry
about branching.

By \cref{thm:ideal}, there is a homeomorphism that puts $\Fcs$ into ideal position.
Therefore, we assume in this section without loss of generality that $\Fcs$
is already in ideal position.

We will explain how to use the center flow $\phic$
to construct a vector field on the surface $S.$
If there are no vertical center circles, then this vector field will be
non-zero everywhere and since $\operatorname{genus}(S) \ge 2,$
this will give a contradiction.
In later sections when we consider vertical center circles,
the averaged vector field will have zeros and
further steps will be needed
to arrive at a contradiction.

\smallskip{}

Much of the work in this section will be on the intermediate covering
space $\hat M$ introduced in \cref{sec:intercover}.
%
We will use the symbol $\pi$ to denote both the projection $\pi : M \to S$
defining $M$ as a circle bundle over the surface $S$ and
the projection $\pi : \hat M \to \bbH^2$
defining $\hat M$ as a circle bundle over $\bbH^2.$
We leave the covering maps $\hat M \to M$ and $\bbH^2 \to S$ unnamed.
The foliations $\Fcs$ and $\Fc$ on $M$ lift to foliations $\Fcshat$ and $\Fchat$
on $\hat M.$ The center flow $\phic$ lifts to a flow $\hatphic$ on $\hat M.$

\begin{figure}
    \centering
    \includegraphics{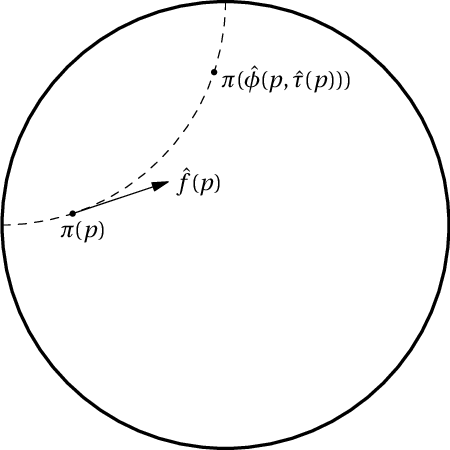}
    \caption[Geod]{ The map $\hat f : \hat M \to T \bbH^2$
    at a point $p \in \hat M.$
    This map is quotiented to produce the induced map
    $f : M \to T S.$ }
    \label{fig:geod}
\end{figure}
Let $\tau : M \to \bbR$ be a continuous positive function on $M.$
This lifts to a function $\hat \tau : \hat M \to \bbR$
by composing with the covering map $\hat M \to M.$
Consider a point $p \in \hat M$ and let $q = \hatphic(p, \hat \tau(p)).$
That is, $q \in \hat M$ is the point obtained by
flowing along the center foliation by time $\hat \tau(p).$
Then $\pi(p)$ and $\pi(q)$ are points in $\bbH^2.$
Let $\exp_{\pi(p)} : T_{\pi(p)} \bbH^2 \to \bbH^2$
denote the exponential map at the point $\pi(p).$
This map is bijective and so there is a unique vector
$v \in T_{\pi(p)} \bbH^2$ such that
\[
    \exp_{\pi(p)}(v) = \pi(q).
\]
We define a map $\hat f : \hat M \to T \bbH^2$
by setting $\hat f(p) = v.$
That is, $\hat f$ is implicitly defined as
\[
    \exp_{\pi(p)}(\hat f(p)) = \pi(\hatphic(p, \hat \tau(p)) )
\]
for all points $p \in \hat M.$ See \cref{fig:geod}.
Note that $\hat f(p)$ will be the zero vector of $T_{\pi(p)} \bbH^2$
if and only if $\pi(p) = \pi(q)$ in $\bbH^2.$

As $\hat \tau,$ $\hatphic,$ and the exponential map
are all continuous, $\hat f$ is continuous.
Moreover, one can see that
$\hat f$ is equivariant with respect to deck transformations.
Specifically,
if $\Gam : \hat M \to \hat M$
is a deck transformation of the covering $\hat M \to M,$
then there is a deck transformation $\gam : \bbH^2 \to \bbH^2$
of the covering $\bbH^2 \to S$ such that $\pi \Gam = \gam \pi$
and consequently
$\hat f(\Gam(p)) = D \gam(\hat f(p))$
where $D \gam : T \bbH^2 \to T \bbH^2$ is the derivative of $\gam.$
Because of this equivariance,
$\hat f$ quotients to a map $f : M \to TS$
where $TS$ is the tangent bundle of the surface.
We call $f$ the \emph{induced map};
it is induced by both the center flow $\phic$
and the choice of flow time $\tau.$

Here, we hope that the use of the letter $f$ does
not cause any confusion.
In the abstract setting of \cref{thm:twofolns},
we no longer have a partially hyperbolic map
$f : M \to M$ and can reuse the letter $f$ for another purpose.

So far, we have not used that $\Fcs$ is in ideal position.
We do so now.

\begin{lemma} \label{lemma:fzero}
    For a point $p \in M,$
    the vector $f(p) \in T_{\pi(p)} M$
    is a zero vector if and only if
    $p$ lies in a vertical leaf of $\Fcs$ and
    $\phic(p, \tau(p))$ lies on the same fiber as $p.$
\end{lemma}
\begin{proof}
    Though the lemma is stated for points on $M,$
    we initially work on the covering space $\hat M.$
    Let $p$ be a point on $\hat M$
    and let $L \in \Fcshat$ be the leaf through $p.$
    Consider for now the case where $L$ is a horizontal leaf.
    Since the foliation is assumed to be in ideal position,
    \cref{lemma:horizontalplane} implies that $L$ is a topological disk
    and $\pi|_L$ is a homeomorphism to its image.
    If the center flow on $L$ had a periodic orbit,
    the orbit would bound a disk containing a fixed point,
    but the center flow has no fixed points.
    This implies for all $t > 0$ that
    $\hatphic(p, t) \ne p$ and the injectivity of $\pi|_L$
    implies that
    $\pi \hatphic(p, t) \ne \pi(p)$
    and therefore $\hat f(p)$ is non-zero.

    Now consider the case where $L$ is a vertical leaf
    of $\Fcshat.$
    Then $L = \pi \inv(\gam)$ where $\gam$ is a complete geodesic in $\bbH^2.$
    Clearly, $\hat f(p)$ is a zero vector
    if and only if $p$ and $\hatphic(p, \hat \tau(p))$
    lie on the same fiber.
    This finishes our consideration on $\hat M,$ but the lemma
    is stated for points in $M,$ not $\hat M.$
    To complete the proof, note that as $\Fcs$ has no compact leaves,
    $\gam \subof \bbH^2$ cannot project to a closed geodesic on $S.$
    Therefore
    the covering maps $\bbH^2 \to S$ and $\hat M \to M$
    must be injective when restricted to $\gam$ and $L$ respectively.
\end{proof}
\begin{lemma} \label{lemma:goodtime}
    If the center foliation $\Fc$ has no vertical center circles
    and no one-ended leaves, then there is a choice of flow time
    $\tau$ such that the induced map $f$ is nowhere zero.
\end{lemma}
The condition that $\Fc$ has no one-ended leaves
is already included in the assumptions of \cref{thm:twofolns},
but we state it in the lemma again for clarity.

\begin{proof}
    We will in fact find a constant $T > 0$ such that
    if $\tau(p) \ge T$ for all $p \in M,$
    then $f$ is nowhere zero.
    Assume no such $T$ exists. Then 
    there are sequences of points $\{p_n\}$ in $M$ and times $\{t_n\}$ in $\bbR$
    such that the times $t_n$ tend to $+\infty$
    and for every $n,$
    the points $p_n$ and $\phi(p_n, t_n)$ lie on the same fiber.

    For each $n,$ the segment $J_n$ = $\phi(p_n, [0, t_n])$
    lies in a leaf $L_n$ of the vertical lamination $\Lamcs.$
    Let $\diam(J_n)$ be the diameter of $J_n$
    with respect to the metric inside $L_n.$
    That is, any two points of $J_n$ can be connected by a path in $L_n$
    of length at most $\diam(J_n)$.

    We claim that $\diam(J_n)$ tends to infinity.
    Indeed, suppose that $\diam(J_n)$ were bounded.
    By going to a subsequence, assume that sequence
    $z_n$ = $\phi(p_n, \tfrac{t_n}{2})$ converges to a point $z \in M.$
    Both the center leaves $\Fc(z_n)$ and the center segments $J_n$
    converge in the compact-open topology
    to the center leaf $\Fc(z),$ showing that $\Fc(z)$ has finite diameter
    inside of a vertical cs-leaf.
    By \cref{thm:cylinderfoliation},
    the closure of $\Fc(z)$ contains a circle leaf,
    which contradicts the assumption of no vertical center circles.
    This proves the claim.

    Define $r_n = \phi(p_n, t_n)$ so that $p_n$ and $r_n$ are the two
    endpoints of $J_n.$
    Let $q_n$ be a point in $J_n$ which lies on the fiber furthest away
    from the fiber containing $p_n$ and $r_n.$
    See \cref{fig:pqr}.
    \begin{figure}
        \centering
        \includegraphics{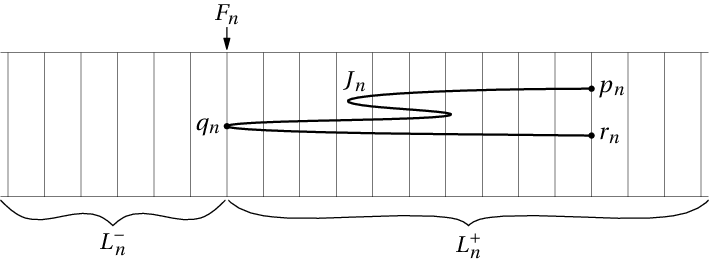}
        \caption{
        A center segment $J_n$ inside of a vertical
        cs leaf $L_n$ as in the proof of \cref{lemma:goodtime}.
        The top of the figure is identified with the bottom,
        so that each vertical line represents a circle in the circle fibering.
        }
        \label{fig:pqr} \end{figure}
    Note that the distances $d(p_n, q_n)$ and $d(q_n, r_n)$ both tend to infinity.
    Define times $s_n \in [0, t_n]$ so that
    $q_n = \phi(p_n, s_n).$
    Then the sequences
    $\{s_n\}$ and $\{t_n - s_n\}$ both tend to infinity.

    If we cut $L_n$ along the fiber
    $F_n$ = $\pi \inv \pi(q_n)$ containing $q_n,$
    then it divides $L_n$ into two half leaves $L_n^-$ and $L_n^+$
    where
    $\del L_n^- = \del L_n^+ = F_n.$
    We assume each of $L_n^+$ and $L_n^-$ contains the boundary fiber $F_n$
    and so one of
    $J_n \subof L_n^+$ or $J_n \subof L_n^-$
    holds.

    By passing to a subsequence, assume that $q_n$ converges to a point
    $q \in M.$ This point lies on a leaf $L \in \Lamcs.$
    By cutting $L$ along the fiber containing $q,$
    we can split $L$ into half leaves $L^-$ and $L^+.$
    We assume that the half leaves are labelled in such a way
    that $L_n^+$ converges to $L^+$ in the compact-open topology
    and similarly $L_n^-$ converges to $L.$

    By passing to a subsequence, we may assume that either
    $J_n \subof L_n^+$ or $J_n \subof L_n^-$ holds for all $n.$
    Assume the former.
    Then $J_n$ converges in the compact-open topology
    to the center leaf $\Fc(q)$ through $q$ and
    since $J_n \subof L_n^+$ for all $n,$
    it follows that $\Fc(q) \subof L^+.$
    By \cref{thm:cylinderfoliation},
    the closure of $\Fc(q)$ must contain
    either a vertical circle or a one-ended leaf (or both)
    and this contradicts our assumptions on $\Fc.$
\end{proof}

If the induced map $f : M \to TS$ is nowhere zero,
we can normalize the vectors to yield a map
\[
    f_1 : M \to T^1 S, \quad p \mapsto \frac{f(p)}{ \| f(p) \| }
\]
where $T^1 S$ is the unit tangent bundle.
Notice that the following diagram commutes:
\[
    \simpleCD
    {M}{f_1}{T^1 S}
    {}{}{}
    {S}{\operatorname{id}}{S}
\]
where the vertical arrows are the projections
which define the circle fiberings on $M$ and $T^1 S.$
Maps of this nature are studied in detail in
\cite{ham2020horizontal}.
See Section 4 of that paper in particular.
We will use the following facts:
\begin{enumerate}
    \item If $C$ is a fiber of $M,$ then
    the degree $d$ of the map $f_1|_C$
    is constant and independent of the fiber $C.$
    \item
    If $d \ne 0,$ then $f_1$ is homotopic to a covering map $M \to T^1 S.$
    \item
    If $d = 0,$ then $f_1$ is homotopic to a composition
    $M \xrightarrow{\pi} S \xrightarrow{X} T^1 S$
    where $X : S \to T^1 S$ is a unit vector field.
    We call $X$ the \emph{averaged vector field}.
\end{enumerate}
The averaged vector field is constructed by
averaging the angles of the unit vectors,
not by averaging the vectors themselves.
Again, see \cite{ham2020horizontal} for details.

\begin{proof}
    [Proof of \cref{thm:twofolns} in the case of no vertical center circles.]

    Let $\tau : M \to \bbR$ be given as in \cref{lemma:goodtime},
    so that the induced map $f : M \to TS$ is nowhere zero.
    Consider on $\hat M$ a fiber $C$ inside of a vertical leaf $L.$
    As $L$ is vertical and the foliation is in ideal position,
    there is a geodesic $\ell \subof \bbH^2$ such that $L = \pi \inv(\ell).$
    For all points $p \in C,$ both $\pi(p)$ and $\pi(\hatphic(p, \hat \tau(p) ) )$
    lie on the same geodesic $\ell,$
    and so the vector $\hat f(p) \in T_{\pi(p)} \bbH^2$ is tangent to this geodesic.
    This implies that $\hat f_1|_C$ is not surjective
    as a map from $C$ to the set of unit tangent vectors $T^1_{\pi(p)} \bbH^2$
    and therefore the degree is $d = 0.$
    As explained above,
    this implies that there is an averaged vector field $X : S \to T^1 S,$
    but as $S$ is a higher-genus surface with non-zero 
    Euler characteristic, this gives a contradiction.
\end{proof}
This completes the proof in the case of no vertical center circles.
The remaining sections of this paper all contribute to the proof
of \cref{thm:twofolns} in the presence of such circles.
The strategy is to first apply an isotopy so that
all of the center circles are fibers.
Then by adapting the above ideas,
we produce an averaged vector field
$X : S \to T S$ where now $X(p) = 0$ if and only if
$\pi \inv(p)$ is a center circle.
Of course, it is entirely possible for a higher-genus surface
to have a vector field with finitely many zeros,
and so a more sophisticated analysis of $X$ is needed
to arrive at a contradiction.

\section{Regularity near a center circle} \label{sec:regularity} 

We now consider \cref{thm:twofolns} in
the case where $\Fc$ contains vertical circles.
Here, we have to be slightly more careful when putting $\Fcs$ into
ideal position.
Let $\Lamcs$ be the lamination of vertical leaves.
A leaf $L \in \Lamcs$ is a \emph{critical leaf} if it contains a center circle.
By assumption, there are only finitely many critical leaves.

\begin{lemma} \label{lemma:critiso}
    If a leaf is critical, then it is isolated.
\end{lemma}
\begin{proof}
    Suppose $\Lcs \in \Lamcs$ is a critical, non-isolated leaf.
    Let $C$ be the center circle in $\Lcs$ and
    let $\Lcu \in \Fcu$ be the cu leaf through $C.$
    The intersections of $\Lcu$ with $\Lamcs$ accumulate on $C$ and
    by the hypotheses of \cref{thm:twofolns},
    each of these intersections is a distinct center circle,
    a contradiction.
\end{proof}
The proof of \cref{thm:ideal} given in the appendix
defines an isotopy to put $\Fcs$ in ideal position.
The result is an isotoped foliation $h(\Fcs)$ such that:
\begin{itemize}
    \item The vertical sublamination is the pull back by $\pi : M \to S$
    of a geodesic lamination on the hyperbolic surface $S.$
    Therefore, this sublamination is smooth.
    \item
    The horizontal leaves are transverse to the fibers.
    In particular, the horizontal leaves are tangent
    to a plane field and this plane field is continuous
    when restricted to the open set $M \sans h(\Lamcs).$
    \item
    Together, the last two items imply that $h(\Fcs)$ is at
    every point in $M$ tangent to a plane field.
    However, the plane field might not vary continuously on all of $M.$
    In particular, there may be a sequence
    $\{x_n\}$ of points on horizontal leaves
    converging to a point $x$ on a vertical leaf,
    but such that their tangent planes
    $\Ecs_{x_n} \subof T_{x_n}M$ do not converge to the tangent plane
    $\Ecs_x \subof T_x M$ of the limit point.
\end{itemize}
This issue arises as the Douady-Earle map used
in the proof of \cref{thm:ideal}
to straighten out curves to geo\-de\-sics might not be $C^1$ on its boundary.
In a neighborhood of a vertical circle, we want continuity of
the tangent planes and so we prove the following upgraded version
of \cref{thm:ideal}.

\begin{addendum} \label{addendum:superDE}
    Let $\Fcal$ be a foliation which satisfies the hypotheses of \cref{thm:ideal}.
    Let $L_1, \ldots, L_n$ be isolated vertical leaves
    and for each $i,$ let $K_i$ be a compact subset of $L_i.$
    Then there is a homeomorphism $h$ isotopic to the identity such that
    $h(\Fcal)$ is in ideal position and $h$ is a $C^1$-diffeomorphism
    in a neighborhood of each $K_i.$
\end{addendum}
As with \cref{thm:ideal}, this addendum is proved in \cref{sec:ideal}.

\section{Indices and the Poincar\'e--Hopf theorem} \label{sec:indices} 

The Poincar\'e--Hopf theorem relates the Euler characteristic of a surface $S$
to the indices of the isolated zeros of a vector field $X$ on that surface.
We will use a specialized version of the Poincar\'e--Hopf theorem
for a surface with piecewise $C^1$ boundary and allowing zeros on the boundary.

We first consider the case where the vector field $X$ integrates to a
continuous flow $\phi.$
We use the notation $\phi^t(x)$ and $\phi(x,t)$ interchangeably for the flow.

Assume that a surface $S$ has a piecewise $C^1$ boundary where
each $C^1$ segment in the boundary is either
tangent to the flow or transverse to the flow.
We call a segment $\sig$ an \emph{outflow edge} if
it is transverse to the vector field and
there is a small neighbourhood $U$ of $\sig$ in $S$ such that for $x \in U,$
the forward orbit through $x$ exits the surface through $\sig.$
To be precise, for every $x \in U$ there is $t \ge 0$ such that
$\phi(x, [0,t]) \subof U$ and $\phi^t(x) \in \sig.$

The definition of an \emph{inflow edge} is analogous:
there is a small neighbourhood $U$ of $\sig$ in $S$ such for every $x \in U$
there is $t \le 0$ such that
$\phi(x, [t,0]) \in U$ and $\phi^t(x) \in \sig.$
In figure \ref{fig:region},
the boundary has two outflow edges on the left and two inflow edges on the
right.

If $X$ does not integrate to a flow, we define outflow edges
by looking at all integral curves.
An edge $\sig$ is an outflow edge if it is transverse to $X$ and there is
a small neighborhood $U$ such that any integral curve
starting at a point $x \in U$ and tangent to $X$ must remain in $U$
until it hits $\sig.$
Inflow edges are defined analogously.

If $p \in \del S$ and $X(p)$ is zero, we define the index of $X$ at $p$
by doubling the surface along the edge
containing $p,$ calculating the index for $p$ in this doubled surface
and then dividing by two.
For instance,
a ``half-saddle'' on the boundary has index $-1/2$ and
a ``half-sink'' or ``half-source'' has index $+1/2$.

When we say that a piece $\sig$ of the boundary is tangent to the flow $X,$
this allows the possibility of $X$ being zero at points along $\sig.$

\begin{thm} \label{thm:poincarehopf}
    Let $S$ be a surface with piecewise $C^1$ boundary
    and $X$ be a vector field defined on $S$ with the following properties:
    \begin{enumerate}
        \item $X$ has finitely many zeros;
        \item
        on the $C^1$ pieces of each boundary component of $S,$
        $X$ alternates between being tangent to $\del S$
        and being transverse to $\del S;$
        \item
        each transverse piece is either an inflow edge or outflow edge; and
        \item
        zeros of $X$ can occur on the boundary of $S,$
        but only in the interior of a $C^1$ piece where
        $X$ is tangent to $\del S.$
    \end{enumerate}
    In this setting, the Euler characteristic of $S$ satisfies
    \[
        \chi(S)
        \ = \ 
        \frac{1}{2} N_{\pitchfork}
        \ + \!\!\!
        \sum_{X(p) = 0} \!\! \operatorname{index}(p, X)
    \]
    where $N_{\pitchfork}$
    is the number of boundary pieces transverse to $X.$
\end{thm}
Note that for this version of the theorem,
we do not allow a boundary component to be everywhere tangent
or everywhere transverse.
The above assumption that
$X$ alternates between
being tangent to $\del S$ and transverse to $\del S$ means that
each boundary component must have at least two $C^1$ pieces
and that the total number of tangent pieces
is equal to the total number of transverse pieces.
See figure \ref{fig:region} for an example.

The above version of the Poincar\'e--Hopf theorem
can be proved from the standard version
by making two copies of $S$ and gluing them along the boundary segments
tangent to the flow.
To apply the theorem, we will also need the following result
in order to calculate indices.

\begin{figure}
    \centering
    \includegraphics{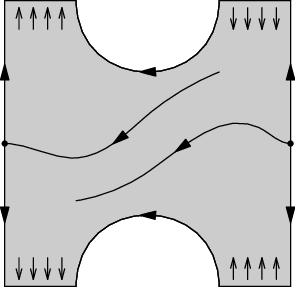}
    \caption{
    A surface with boundary for which \cref{thm:poincarehopf} applies.
    The piecewise $C^1$ boundary has eight pieces, four of which are transverse
    to the flow.
    The boundary has two half-saddles and we assume that there are no other
    fixed points.
    The Euler characteristic is therefore
    $\chi(S) = \tfrac{1}{2} 4 + (\tfrac{-1}{2}) + (\tfrac{-1}{2}) = 1$.}
    \label{fig:region}
\end{figure}
\begin{prop} \label{prop:productmap}
    Suppose $f,g : [-1,1] \to \bbR$ are continuous,
    strictly-increasing functions
    such that zero is the unique fixed point of both $f$ and $g.$
    Let $h$ be a homeomorphism of $\bbR^2$ such that
    $h(x,0) = (x,0)$ for all $x \in [-1,1]$
    and $h(0,y) = (0,y)$ for all $y \in [-1,1].$
    Define $r : h([-1,1]^2) \to \bbR^2$ by
    $r(h(x,y)) = h(f(x), g(y)).$

    For an integer $n \ge 1,$ choose a neighborhood $U$ of $(0,0) \in \bbR^2$
    such that $r^n$ is well defined on $U$
    and 
    define a vector field $X : U \to \bbR^2$
    by $X(p) = r^n(p) - p.$

    If $n$ is sufficiently large,
    then
    \begin{enumerate}
        \item $X$ is non-zero everywhere except at the origin,
        \item
        at the origin, the index of $X$ is in $\{-1,0,+1\}$, and
        \item
        this index is independent of $n.$
    \end{enumerate} \end{prop}
\begin{proof}
    Item (1) holds because $r$ is topologically conjugate to
    $f \ti g,$ which has no periodic points other than the origin.
    To prove items (2) and (3), we define a continuous curve
    $Q = Q_1 \cup Q_2 \cup Q_3 \cup Q_4$
    where each $Q_i$ defines the curve on one of the four quadrants of
    $\bbR^2.$
    The index of the vector field $X$ at the origin is given by the
    homotopy type of the map
    $X|_Q : Q \to \bbR^2 \sans \{(0,0)\}.$
    To determine this, we determine for each subcurve $Q_i,$
    the endpoint-fixing homotopy type of
    $X|_{Q_i} : Q_i \to \bbR^2 \sans \{(0,0)\}$
    and show in each quadrant that it corresponds to a
    90 degree rotation of the angle.
    Therefore, the winding number of $X|_Q$ must be in $\{-1,0,1\}.$
    All of the quadrants behave similarly, so we only consider
    the curve $Q_1,$ which is a subset of the quadrant
    $[0, \infty) \ti [0, \infty).$

    Since $f(0) = 0$ is the only fixed point of $f,$
    we either have
    \[
        \lim_{n \to \infty} f^n(1) = 0
        \qorq \lim_{n \to \infty} f \invn(1) = 0.
    \]
    Similarly, $g$ satisfies either 
    \[
        \lim_{n \to \infty} g^n(1) = 0
        \qorq \lim_{n \to \infty} g \invn(1) = 0.
    \]
    Therefore, the proof breaks into four cases.
    We analyze one of these cases in detail and
    outline the approach for the other three.
    The case we consider in detail is where
    \[
        \lim_{n \to \infty} f \invn(1) =
        \lim_{n \to \infty} g^n(1) = 0.
    \]
    Assume these limits hold and therefore $x < f(x)$ and $g(y) < y$
    for all $x$ and $y$ in $(0,1].$
    Then for each $n,$
    \[
        (f \ti g)^n \big( \{ f \invn(1) \} \ti [0,1] \big)
        \ = \
        \{1\} \ti [0, g^n(1)]
    \]
    and therefore
    \[
        r^n h \big( \{ f \invn(1) \} \ti [0,1] \big)
        \ = \
        h \big( \{1\} \ti [0, g^n(1)] \big).
    \]
    Write $h_x$ and $h_y$ for the two coordinates of $h.$
    If we choose $n$ sufficiently large, then
    \[
        h_x( f \invn(1), y) \ < \ h_x(1, g^n(y))
        \qquad \text{for all }
        y \in [0, 1].
    \]
    Similarly,
    \[
        (f \ti g)^n \big( [0, f \invn(1) ] \ti \{1\} \big)
        \ = \ [0, 1] \ti \{g^n(1)\}
    \]
    implies
    \[    
        r^n h \big( [0, f \invn(1)] \ti \{1\} \big)
        =
        h \big( [0,1] \ti \{g^n(1)\} \big)
    \]
    and if $n$ is sufficiently large, then
    \[
        h_y( x, g^n(1))
        <
        h_y( f^n(x), 1)
        \qquad \text{for all }
        x \in [0,1].
    \]
    Now choose $n$ large, and define
    \[
        Q_1 \ = \ h \bigg(
        \big(\{f \invn(1) \} \ti [0,1] \big)
        \ \cup \
        \big([0, f \invn(1)] \ti \{1\} \big) \bigg).
    \]
    The vector field $X$ is defined by $X(q) = r^n(q) - q.$
    Our choice of $n$ implies that at all points $q \in Q_1,$
    either the $x$-coordinate of $r^n(q)$ is greater than
    the $x$-coordinate of $q$ or
    the $y$-coordinate of $r^n(q)$ is less than
    the $y$-coordinate of $q$ (or both).
    This means that $X(q)$ does not take any values in the interior
    of the quadrant $(-\infty, 0] \ti [0, \infty).$
    Therefore, $X|_{Q_1}$ cannot wind around the origin and
    it must be homotopic to a simple 90 degree rotation in angle.

    We now briefly consider the other three cases.
    \begin{itemize}
        \item If $\lim_{n \to \infty} f^n(1) = \lim_{n \to \infty} g \invn(1) = 0,$
        then take
        \[  Q_1 = h \bigg(
            (\{1\} \ti [0, g \invn(1)])
            \cup
            ([0,1] \ti \{g \invn(1)\}) \bigg).  \]
        \item
        If $\lim_{n \to \infty} f^n(1) = \lim_{n \to \infty} g^n(1) = 0,$
        then take
        \[  Q_1 = h \bigg(
            (\{1\} \ti [0, 1])
            \cup
            ([0,1] \ti \{1\}) \bigg).  \]
        \item
        If $\lim_{n \to \infty} f \invn(1) = \lim_{n \to \infty} g \invn(1) = 0,$
        then take
        \[  Q_1 = h \bigg(
            (\{f \invn(1)\} \ti [0, g \invn(1)] )
            \cup
            ([0, f \invn(1)] \ti \{ g \invn(1)\} ) \bigg).
        \] \end{itemize}
    In each case, if $n$ is sufficiently large,
    then $X|_{Q_1}$ avoids the interior of one of the quadrants
    of $\bbR^2$ and this therefore determines the homotopy type
    of $X|_{Q_1}$ based on its values at the two endpoints of $Q_1.$
    Further, whether the rotation is clockwise or counterclockwise
    is purely determined by which of the four cases we are in
    and is independent of $n.$
\end{proof}

We also state a modified version of \cref{prop:productmap}
for a point on the boundary of a surface.

\begin{prop} \label{prop:halfproductmap}
    Suppose $f : [0,1] \to [0, +\infty)$ and $g : [-1,1] \to \bbR$
    are continuous, strictly-increasing functions
    such that zero is the unique fixed point of both $f$ and $g.$
    Let $h$ be a homeomorphism of $[0,\infty) \ti \bbR$ such that
    $h(x,0) = (x,0)$ for all $x \in [0,1]$
    and $h(0,y) = (0,y)$ for all $y \in [-1,1].$
    Define $r : h([0,1] \ti [-1,1]) \to [0, \infty) \ti \bbR$ by
    $r(h(x,y)) = h(f(x), g(x)).$

    For an integer $n \ge 1,$ choose a neighborhood
    $U$ of $(0,0) \in [0, \infty) \ti \bbR$
    such that $r^n$ is well defined on $U$
    and 
    define a vector field $X : U \to \bbR^2$
    by $X(p) = r^n(p) - p.$

    If $n$ is sufficiently large,
    then
    \begin{enumerate}
        \item $X$ is non-zero everywhere except at the origin,
        \item
        at the origin, the index of $X$ is in
        $\{-\frac{1}{2}, 0, +\frac{1}{2}\},$ and
        \item
        this index is independent of $n.$
    \end{enumerate} \end{prop}
\begin{proof}
    We can double the surface $[0, \infty) \ti \bbR$
    to reduce the hypotheses to those of \cref{prop:productmap}.
\end{proof}
Propositions \ref{prop:productmap} and \ref{prop:halfproductmap}
are formulated and proved in the setting of
a Euclidean metric on $\bbR^2,$ but we will actually want to apply them
to a vector field $X$ on the hyperbolic plane, where $X(p)$
points in the direction of a geodesic going from $p$ to $r^n(p).$
We can equip the open unit disk $D \subof \bbR^2$ either with the
Euclidean metric from $\bbR^2$ or with the hyperbolic metric.
For any $\ep > 0,$ there is a neighbourhood $U \subof D$ of the origin
such that if $p$ and $q$ both lie in $U,$ then the Euclidean line from
$p$ to $q$ and the hyperbolic geodesic from $p$ to $q$
differ in angle by less than $\ep.$
Therefore, the above results also hold in the hyperbolic setting.

\section{A good neighborhood of plaques} 

For the following definitions, assume that $\Fcs$ and $\Fcu$ are continuous
foliations on a circle bundle $M$ with fibers given by $\pi : M \to S$
and that $\Fcs$ and $\Fcu$ intersect in a one-dimensional foliation $\Fc.$
We assume that $S$ has constant negative curvature and is a quotient of $\bbH^2.$

Consider a leaf $C \in \Fc$ which is also
a fiber of the fibering;
that is, there is $p \in S$ such that $C = \pi \inv(p).$
Then, a \emph{neighbourhood of plaques}
is a homeomorphism of the form
$i : D \ti S^1 \to \pi \inv(D)$
where $D \subof S$ is a geometric disk centered at $p$
and where $\pi(i(q,z)) = q$ for all $(q,z) \in D \ti S^1.$
Each $z \in S^1$ defines a embedded disk $D_z \subof M$
by $D_z$ = $i(D \ti \{z\}).$
We call each such $D_z$ a \emph{plaque}.

We further say that this is a \emph{good} neighbourhood of plaques
if the following additional properties hold:
\begin{enumerate}
    \item There is a geodesic segment $\gam^s \subof D$ 
    passing through $p$ such that
    if $\Lcs$ denotes the leaf of $\Fcs$ containing $C,$ then
    \[
        \Lcs \cap \pi \inv (D) = \pi \inv(\gam^s).
    \]
    \item
    Similarly,
    there is a geodesic segment $\gam^u \subof D$ 
    passing through $p$ such that
    if $\Lcu$ denotes the leaf of $\Fcu$ containing $C,$ then
    \[
        \Lcu \cap \pi \inv (D) = \pi \inv(\gam^u).
    \]
    \item
    The geodesics segments $\gam^u$ and $\gam^s$ intersect at right angles at $p.$
    \item
    The intersection of $\Fcs$ with the foliation $\{ D_z : z \in S^1 \}$
    produces a one-dimensional foliation $\Fs$ on $\pi \inv(D).$
    \item
    Similarly,
    the intersection of $\Fcu$ with the foliation $\{ D_z : z \in S^1 \}$
    produces a one-dimensional foliation $\Fu$ on $\pi \inv(D).$
    \item
    Inside of a single plaque,
    the foliations $\Fs$ and $\Fu$ have local product structure
    in the following sense:
    there is a small neighbourhood $U$ of $p$ such that
    if $q,r \in \pi \inv(U)$ lie on the same plaque,
    then $\Fs(q)$ and $\Fu(r)$ intersect in exactly one point.
    Moreover, the intersection point varies continuously.
\end{enumerate}
We stress that all of the above conditions are $C^0$ in nature.
We do not assume that the foliations have $C^1$ leaves.
We have chosen the names $\Fu$ and $\Fs$ above for convenience,
but note that $\Fu$ and $\Fs$ in this context are not the stable
or unstable foliations of a partially hyperbolic map.

\begin{prop} \label{prop:straighten}
    Suppose $\Fcs_0$ and $\Fcu_0$ are foliations which satisfy
    the hypotheses of \cref{thm:twofolns}.
    Then there is a homeomorphism $h_1 : M \to M$ isotopic to the identity
    such that the isotoped foliations
    $\Fcs_1$ = $h_1(\Fcs_0)$ and $\Fcu_1$ = $h_1(\Fcu_0)$
    satisfy the following properties:
    \begin{enumerate}
        \item the cs foliation is $\Fcs_1$ is in ideal position,
        \item
        every center circle $C$ in $\Fc_1$ = $\Fcs_1 \cap \Fcu_1$
        is a fiber,
        that is, $C = \pi \inv(p)$ for some $p \in S,$ and
        \item
        every center circle $C$ in $\Fc_1$ has
        a good neighborhood of plaques.
    \end{enumerate} \end{prop}
Before proving the proposition, we first discuss return maps to plaques.
Let $\phi$ be a continuous flow on $M$ whose orbits are exactly the leaves of $\Fc.$
Let $\{ D_z : z \in S^1 \}$ be a good neighborhood of plaques
associated to a center circle $C = \pi \inv(p_0).$
Say each plaque $D_z$ projects to $D \subof S.$
Let $n \ge 1$ be an integer.
Then there is a small compact disk $K \subof D$ much smaller than $D$ and
centered at $p_0 \in S$ such that the following holds.
For any point $p \in \pi \inv(K) \cap D_z \subof M,$
the forward orbit $\phi(p, (0,\infty))$ does not leave $\pi \inv(D)$ before
intersecting $D_z$ at least $n$ times.

We can then define the \emph{$n$-th return time to the plaque}
as a function $\tau : \pi \inv(K) \to (0, \infty)$ where
$\phi( p, (0, \tau(p)] )$ intersects the plaque through $p$
exactly $n$ times,
including at $\phi(p, \tau(p)).$
Associated to the $n$-th return time is the
\emph{$n$-th return map to the plaque} defined by
\[
    r : \pi \inv(K) \to \pi \inv(D), \quad p \mapsto \phi(p, \tau(p)).
\]
Note that $r$ maps each set $\pi \inv(K) \cap D_z$ 
into $D_z.$
In this paper,
return times and return maps always refer to the definitions above.

\smallskip{}

The rest of the section is dedicated to the proof of \cref{prop:straighten}.
First, we give a lemma about cylinders which will be used inside of a single
vertical cs leaf.
In what follows, the word \emph{smooth} denotes the $C^\infty$ category.

\begin{lemma} \label{lemma:straightcylinder}
    Suppose $\al$ is a vertical circle in $\bbR \ti S^1$
    which is $C^1$-regular and is disjoint from $\{0\} \ti S^1.$
    Then, there is $\ep > 0$ and a smooth embedding
    $j : (-\ep, 1) \ti S^1 \to \bbR \ti S^1$
    with the following properties.
    \begin{enumerate}
        \item $j(x,z) = (x,z)$ for all $(x,z) \in (-\ep,+\ep) \ti S^1,$
        \item
        $\al$ is contained in $j( (-\ep, 1) \ti S^1 ),$ and
        \item
        for each $z \in S^1,$ the curve $j( (-\ep,1) \ti \{z\} )$
        intersects $\al$ in a unique point and
        this intersection is transverse.
    \end{enumerate} \end{lemma}
\begin{proof}
    For concreteness, assume $\al$ is contained in $(0, \infty) \ti S^1.$
    Choose a constant $r > 0$ such that $\al$ is disjoint from
    $[0, r) \ti S^1,$ and define $\delta = r/2.$
    Define a unit vector field $X$ on $(-r,+r) \ti S^1$
    which is pointing in
    the positive direction of the $\bbR$ coordinate of $\bbR \ti S^1.$
    In other words, the integral curves of $X$ are
    curves of the form $(-r, r) \ti \{z\}$ for $z \in S^1.$

    Define a smooth, non-zero vector field $Z$
    in a small neighborhood $U$ of $\al.$
    After replacing $U$ by an even smaller neighborhood,
    we may assume that every integral curve of $Z$ is a short
    curve which transversely intersects $\al$ at a unique point.
    In effect, $U$ and $Z$ define a tubular neighbourhood of $\al,$
    but $Z$ is smooth whereas $\al$ is only $C^1.$

    Let $\bt \subof U$ be a smooth vertical circle close to $\al,$
    but disjoint from $\al$ such that
    each integral curve of $Z$ intersects $\bt$ transversely
    and at a single point.
    Moreover, assume $\bt$ is on the side of $\al$ closer to $0 \ti S^1.$
    Then, going from ``left to right,''
    we have the vertical circles
    \[
        0 \ti S^1, \quad \delta \ti S^1, \quad \bt, \qandq \al.
    \]
    The circles $\delta \ti S^1$ and $\bt$ bound a set $A \subof \bbR \ti S^1$
    which is homeomorphic to an annulus.
    Since topological surfaces have a unique smooth structure,
    $A$ is diffeomorphic to $[0,1] \ti S^1.$
    Moreover, there is a smooth embedding
    $i : [0,1] \ti S^1 \to \bbR \ti S^1$
    such that $A$ is the image of $i$ and the derivative of $i$
    is well-defined and invertible for all points in $A,$
    even those on the boundary.
    See for instance \cite{hat2026kirby}
    for details.

    Using $i,$ we can define a smooth non-zero vector field $Y$ on $A$
    which is transverse to the boundary.
    In particular, $Y$ points into $A$ along $\delta \ti S^1$
    and points out of $A$ along $\bt.$
    By smoothly interpolating $X$ and $Y$ in a neighborhood of $\delta \ti S^1$
    and smoothly interpolating $Y$ and $Z$ in a neighborhood of $\bt,$
    we can produce a smooth non-zero vector field on all of
    $((-r, +r) \ti S^1) \cup A \cup U$
    which is equal to $X$ in a neighborhood of $0 \ti S^1$
    and equal to $Z$ in a neighborhood of $\al.$
    Integrating this vector field yields the function $j$
    specified in the statement of the lemma.
\end{proof}
We now use this lemma to prove \cref{prop:straighten}.
Let $\Fcs_0$ and $\Fcu_0$ be as in the hypotheses of the proposition
and let $\Fc_0$ denote their intersection.
We construct the homeomorphism $h_1$ of the proposition
as a composition $h_b \circ h_a$ where $h_a$ puts
the cs foliation into ideal position, and then $h_b$ preserves this
foliation and works inside of the cs leaves
to straighten the vertical circles into fibers.
For simplicity, assume that $\Fc_0$ contains a single center circle
$C_0.$ The proof adapts easily to the case of multiple circles.
The $L_0$ denote the leaf of $\Fcs_0$ containing $C_0.$

Let $h_a : M \to M$ be an isotopy such that $\Fcs_a = h_a(\Fcs_0)$ is
in ideal position.
Similarly, define $\Fcu_a = h_a(\Fcu_0),$
$\Fc_a = h_a(\Fc_0),$
$C_a = h_a(C_0),$ and $L_a = h_a(L_0).$
By \cref{addendum:superDE}, we may assume that $h_a$ is $C^1$
in a neighbourhood of $C_0$ and that this neighbourhood
contains a large compact subset of $L_0.$
Therefore, the isotoped foliations $\Fcs_a$ and $\Fcu_a$
are tangent to continuous plane fields in a neighborhood of $C_a.$
Moreover, these two plane fields are transverse in this neighborhood.

Choose a coordinate chart $\psi : [-1,1]^2 \to S$ such that
\begin{itemize}
    \item $\psi(\{0\} \ti [-1,1])$ is a geodesic segment
    contained in $\pi(L_a),$
    \item
    $\pi(C_a)$ is contained in $\psi(\{0\} \ti (0,1)),$ and
    \item
    $\psi([-1,1] \ti \{0\})$ is a geodesic segment
    which intersects $\psi(\{0\} \ti [-1,1])$ at right angles.
\end{itemize}
This chart can be constructed using the exponential map on $S.$
Once $\psi$ is defined, use it to define a smooth embedding
$\Psi : [-1,1]^2 \ti S^1 \to M$
which maps each circle $\{x\} \ti \{y\} \ti S^1$ to the fiber
$\pi \inv(\psi(x,y)).$
By our use of \cref{addendum:superDE},
we may assume that $\Fcs_a$ and $\Fcu_a$ have
continuous tangent plane fields in a neighborhood of
$\Psi(\{0\} \ti [-1,1] \ti S^1).$
Use $\Psi$ to pull back the foliations to $[-1,1]^2 \ti S^1.$
That is, define foliations on $[-1,1]^2 \ti S^1$ by
$\Fcs = \Psi \inv(\Fcs_a)$ and $\Fcu = \Psi \inv(\Fcu_a).$
For the rest of this section, we will work with the foliations
inside the space $[-1,1]^2 \ti S^1.$

Define $C = \Psi \inv(C_a)$ and note that $C \subof \{0\} \ti (0,1) \ti S^1.$
Also note that $\{0\} \ti [-1,1] \ti S^1$ is a leaf of
the pulled back foliation $\Fcs.$
By \cref{lemma:straightcylinder}, there is $\ep > 0$ and
a family $\{J_z : z \in S^1 \}$ of $C^1$ curves
with the following properties:
\begin{itemize}
    \item $J_z$ is contained in $\{0\} \ti [-1,1] \ti S^1,$
    \item
    $J_z \cap (\{0\} \ti [-\ep,\ep] \ti S^1) = \{0\} \ti [-\ep,\ep] \ti \{z\},$
    \item
    $J_z$ intersects $C$ exactly once
    and this intersection is quasi-transverse, and
    \item
    every point of $C$ lies in exactly one curve $J_z.$
\end{itemize}
For each specific value $z_0 \in S^1,$ define a surface
\[
    S_{z_0} \ = \ \big\{ \ (x,y,z) \in [-\ep,\ep] \ti [-1,1] \ti S^1 \ : \ 
    (0,y,z) \in J_{z_0} \ \big\}.
\]
This produces a family of surfaces $\{ S_z : z \in S^1\}.$
Up to replacing $\ep$ with a smaller value,
we may assume that $\Fcs$ is transverse to $S_z$ for all $z \in S^1.$

For now, just consider one specific value of $z \in S^1.$
Let $\Fs$ denote the intersection of $\Fcs$ with $S_z;$
it is a continuous 1-dimensional foliation with $C^1$ leaves
defined on all of $S_z.$
Let $p = (0,0,z)$ and let $q$ be the unique intersection of $J_z$ with
the circle $C.$
Since $S_z$ is transverse to $C$ and $C$ lies inside a leaf of $\Fcu,$
there is a neighborhood $U \subof S_z$ of $q$
such that $\Fcu$ is transverse to $S_z$ at all points in $U.$
Let $\Fu$ denote the intersection of $\Fcu$ with $U;$
it is a continuous 1-dimensional foliation with $C^1$ leaves
defined on $U.$
Since $\Fcs$ and $\Fcu$ are transverse, the 1-dimensional foliations
$\Fs$ and $\Fu$ are transverse on $U,$ and therefore have local product structure
in a small neighborhood of $q.$
Let $\Fu(q)$ denote the leaf of $\Fu$ through the point $q$
and use similar notion for leaves through other points.

Choose $\delta > 0$ much smaller than $\ep.$
In particular, $\delta$ should be chosen small enough that
the segment $[-\delta, \delta] \ti \{0\} \ti \{z\}$
is transverse to $\Fs$ and for any point 
$p' \in [-\delta,\delta] \ti \{0\} \ti \{z\}$
the leaf $\Fs(p')$ through $p'$ intersects
$\Fu(q)$ in a single point, $q'.$
See figure \ref{fig:surfacez} for a depiction of these points.

\begin{figure}
    \centering
    \includegraphics{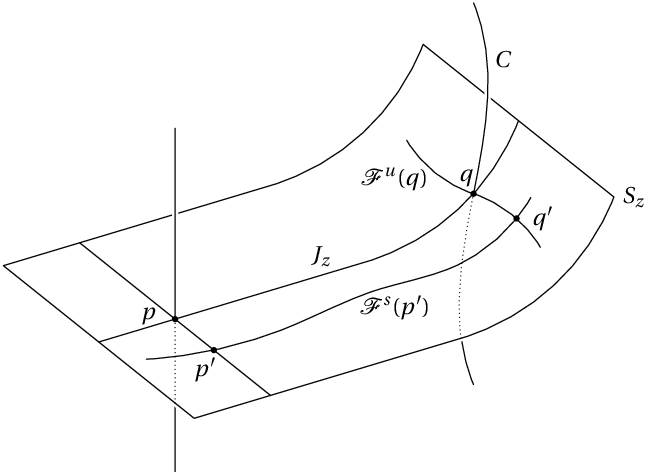}
    \caption{
    A depiction of the curve $J_z$ and the corresponding surface $S_z$.
    The vertical line through the point $p$ depicts the circle
    $\{0\} \ti \{0\} \ti S^1$ where the top and bottom points are identified.
    Similarly, the top and bottom points of $C$ are identified.}
    \label{fig:surfacez}
\end{figure}
We now define a homeomorphism $h : S_z \to S_z$
with the following properties:
\begin{itemize}
    \item $h$ is isotopic to the identity,
    \item
    $h$ equals the identity on a neighborhood of the boundary of $S_z,$
    \item
    $h$ preserves the $\Fs$ foliation, and
    \item
    $h$ maps a small segment of $\Fu(q)$ to $[-\delta,\delta] \ti \{0\} \ti \{z\}.$
\end{itemize}
The last two items mean that $h$ maps each $q'$ as above
to its corresponding $p'.$

We can define such a function $h$ by flowing along the leaves
of the foliation $\Fs.$
First choose an orientation of $\Fs$ and then let $\phi^s$
denote the unit-speed flow along the leaves of $\Fs.$
Then the homeomorphism $h$ will be a function of the form
$h(v) = \phi^s(v, t(v))$
where the function $t : S_z \to \bbR$ gives the flow time.
We first define $t$ on $\Fu(q)$ so that $h$ maps a subset of $\Fu(q)$
to $[-\delta,\delta] \ti \{0\} \ti \{z\}.$
Then we extend $t$ continuously to all of $S_z.$
This can be done in such a way that
the resulting $h$ is a homeomorphism.
Moreover, we can set $t$ equal to zero on 
a neighborhood of the boundary of $S_z$
so that $h$ is the identity on this neighborhood.
The foliations $\Fs$ and $\Fu$ have local product structure
in a neighborhood of $q.$
Therefore, the isotoped foliations $h(\Fs)$ and $h(\Fu)$
have local product structure in a neighborhood of $p = h(q).$

We have described for a single curve $J_z$
how to construct an isotopy on 
its corresponding surface $S_z.$
We can do the same construction for all curves
$\{ J_z : z \in S^1 \},$ producing isotopies
on their corresponding surfaces $\{ S_z : z \in S^1 \}.$
Moreover, these isotopies can be constructed in a way
which varies continuously with $z \in S^1.$
The result is a homeomorphism $h$ defined on all of
$[-1,1]^2 \ti S^1$ which preserves the foliation $\Fcs$
and maps the circle $C$ to $\{0\} \ti \{0\} \ti S^1.$
Since $h$ is the identity near the boundary of $[-1,1]^2 \ti S^1,$
we can use it along with $\Psi : [-1,1]^2 \ti S^1 \to M$ to define
a homeomorphism $h_b : M \to M$
which is the identity outside of the image of $\Psi.$
Then, the composition $h_1 = h_b \circ h_a$
is a homeomorphism which satisfies the conclusions
of \cref{prop:straighten}.

\section{Flow times} 

For this section, assume that in order to prove \cref{thm:twofolns},
we have first applied the isotopy given by \cref{prop:straighten}
to put $\Fcs$ into ideal position
and to make the center circles coincide with fibers.
Further,
each center circle has an associated
``good neighborhood of plaques'' $\{ D_z : z \in S^1 \}.$
After these changes, $\Fcs$ and $\Fcu$ are both $C^0$ foliations
and $\Fcs$ is in ideal position.

The intersection $\Fc = \Fcs \cap \Fcu$ is a $C^0$ foliation and
we use in this section $\phi$ to denote the center flow,
which is now a $C^0$ flow and not necessarily tangent to a vector field.
Recall from \cref{sec:average} that a choice of flow time
$\tau : M \to [0, \infty)$ defines the induced map
$f : M \to TS.$
For each $p \in M,$ we call $f(p)$ the \emph{induced vector} at $p.$

The main goal of this section is to prove the following.

\begin{prop} \label{prop:fullgoodtime}
    With $\Fcs, \Fcu,$ and $\phi$ as above,
    there is a choice of flow time $\tau : M \to [0, \infty)$
    and an integer $n \ge 1$ such that the following hold:
    \begin{enumerate}
        \item at a point $p \in M,$
        the induced vector $f(p) \in T_{\pi(p)} S$ is zero
        if and only if $p$ lies on a center circle, and
        \item
        every center circle $C$ has a neighborhood $U$ such that
        if $p \in U,$ then $\tau(p)$ gives the $n$-th return time
        to the plaque through $p.$
    \end{enumerate} \end{prop}
The remainder of this section is dedicated to the proof of this proposition.
It will be clear from the proof that we can choose the integer $n \ge 1$
as large as desired.
As in the previous section,
we assume for simplicity that there is a unique center circle $C.$
The proof easily generalizes to the case of multiple center circles.

\begin{lemma} \label{lemma:Ugoodtime}
    Let $U \subof M$ be an open set containing the center circle $C.$
    Then there is $T > 0$ such that
    for all $p \in \Lamcs \sans U$ and all $t \in [T, +\infty),$
    the points $p$ and $\phi(p,t)$ lie on distinct fibers.
\end{lemma}
\begin{proof}
    This is an adaptation of the proof of \cref{lemma:goodtime}.
    Let $L \in \Lamcs$ be the vertical cs leaf containing $C.$
    By \cref{lemma:critiso}, $L$ is an isolated leaf of the lamination,
    we can replace $U$ by a smaller neighbourhood and assume that
    $U \cap L$ has a single connected component
    and that no other leaves of $\Lamcs$ intersect $U.$
    We also assume that $U$ is a union of fibers.

    By the last item of \cref{thm:cylinderfoliation},
    there is a smaller neighborhood $V$ of $C$
    with the following property:
    if $p \in V \cap L,$
    then one of the two sets
    $\phi(p, [0, +\infty) )$ or $\phi(p, (-\infty, 0] )$
    is a subset of $U \cap L.$
    Consider $J_n$ as in the proof of \cref{lemma:goodtime},
    but now with endpoints outside of $U.$
    Then $J_n$ is disjoint from $V$ for all $n.$

    If $\diam(J_n)$ were bounded, the $J_n$ would converge to
    a full orbit of bounded diameter outside of $V$
    implying the existence of a center circle outside of $V.$
    Therefore the diameters tend to infinity.
    Define $p_n,$ $q_n,$ and $r_n$ as in the proof of \cref{lemma:goodtime}.
    By passing to a subsequence, $\{q_n\}$ converges to a point $q \in M \sans V$
    and the closure of $\Fc(q)$ lies in a half leaf $L^+$
    which is disjoint from $C.$
    As before,
    the closure of $\Fc(q)$ must contain
    either a vertical circle or a one-ended leaf (or both)
    and this gives a contradiction.
\end{proof}
To choose a good flow time, we define $\tau : M \to [0, \infty)$
so that $\tau$ is given by a constant $T > 0$ far away from the center circles
and so that near a center circle flowing by $\tau$ yields 
the $n$-th return map to a plaque $D_z$ for some large $n.$
We must carefully patch these two definitions together to
define $\tau$ on all of $M.$
Most of the finicky details of this patching are handled
by the following technical lemma about flows on cylinders.

\begin{lemma} \label{lemma:patching}
    Suppose $\phi$ is a flow generated by
    a continuous vector field on $\bbR \ti S^1$ such that
    $\phi$ has no one-ended orbits and
    $0 \ti S^1$ is the unique periodic orbit of $\phi.$
    Further assume
    for every $z \in S^1$ that
    the curve $[-1,1] \ti \{z\}$ is transverse to the flow.
    Let $T > 0.$ Then there exist
    \begin{itemize}
        \item an integer $n \ge 1,$
        \item
        a continuous function $\tau : \bbR \ti S^1 \to [T, +\infty),$
        \item
        a homeomorphism $r : \bbR \ti S^1 \to \bbR \ti S^1,$
        \item
        an open neighborhood $U$ of $0 \ti S^1,$ and
        \item
        a compact set $K \subof \bbR \ti S^1$ \end{itemize}
    with the following properties:
    \begin{enumerate}
        \item the functions $\tau$ and $r$ are related by
        $r(p) = \phi(p, \tau(p))$ for all $p \in \bbR \ti S^1,$
        \item
        $\tau(p) = T$ for all $p \notin K,$
        \item
        if $p \in \big([-1,0) \cup (0,1]\big) \ti S^1$,
        then $p$ and $r(p)$ lie on different vertical fibers, and
        \item
        for any point $p = (x,z) \in U, r(p)$ is equal to the
        $n$-th return of the forward orbit $\phi(p, [0, \infty) )$
        to the segment $[-1,1] \ti \{z\}.$
    \end{enumerate} \end{lemma}
\begin{proof}
    We first consider only the subset $[0, \infty) \ti S^1.$
    Since $0 \ti S^1$ is the only periodic orbit,
    \cref{thm:cylinderfoliation} implies that it
    must either be topologically attracting
    or topologically repelling on $[0, \infty) \ti S^1.$
    We assume for now that it is topologically attracting.
    As there are no one-ended orbits,
    every forward orbit in $(0, \infty) \ti S^1$
    must limit on $0 \ti S^1.$

    By modifying a trapping region near $0 \ti S^1,$
    we can find a global transversal,
    that is, a vertical circle $\al$ which intersects every orbit
    in $(0, \infty) \ti S^1$ exactly once.
    We assume that $\al \subof (0,1) \ti S^1.$

    Let $\pi : \bbR \ti S^1 \to \bbR$ denote projection onto the $\bbR$ coordinate.
    Then $\pi(\al)$ is a compact subset of (0,1).
    For $t \in \bbR,$ let $\phi(\al, t)$ denote the image of $\al$ under the time-$t$ map
    of the flow.
    Choose $t_0 > T$ large enough that the sets
    \[
        \pi \big(\phi(\al, t_0) \big), \quad
        \pi(\al), \qandq
        \pi \big(\phi(\al, -t_0) \big)
    \]
    are pairwise disjoint and such that 
    $\pi \big(\phi(\al, -t_0) \big)$ is a subset of $(1,\infty)$.
    For the remainder of the proof, we use $\al_1$, $\al_2$, and $\al_3$
    to denote
    $\phi(\al, t_0)$, $\al$, and $\phi(\al, -t_0)$ respectively.
    Define
    \begin{itemize}
        \item $A_1$ as the closed annulus between $0 \ti S^1$ and $\al_1,$
        \item
        $A_2$ as the closed annulus between $\al_1$ and $\al_2,$
        \item
        $A_3$ as the closed annulus between $\al_2$ and $\al_3,$ and
        \item
        $A$ as the union $A = A_1 \cup A_2.$
    \end{itemize}    
    For a point $p = (x,z) \in A,$ the forward orbit $\phi(p, [0, +\infty))$
    is contained in $A$ and attracts towards $0 \ti S^1.$
    Therefore, $\phi(p, [0, +\infty))$ intersects the segment
    $[-1,1] \ti z$ infinitely many times.
    Let $n \ge 1$ be a large integer and define
    $r : A \to A$ as the $n$-th return map on each segment $[-1,1] \ti z.$
    Let $\tau : A \to [0, \infty)$ be the associated $n$-th return time,
    so that $r(p) = \phi(p, \tau(p))$ for all $p \in A.$
    We assume that $n$ is chosen large enough that $\tau(p) > 2 t_0$
    for all $p \in A.$

    We now extend $\tau$ to a function on all of $[0, \infty) \ti \bbR.$
    We will do this for each orbit individually, but in such a way
    that it is clearly continuous everywhere.
    Consider a point $p \in \al_2.$
    For $t \in [0, \infty),$ the point $\phi(p,t)$ lies in $A$
    and $\tau(\phi(p,t))$ is already defined.
    Define
    $\tau( \phi(p, -t_0)) = 2 t_0$
    and
    \[
        \tau \big( \phi(p,t) \big) = T
        \text{ for all } t \in (-\infty, -2 t_0).
    \]
    With $\tau(\phi(p,t))$ now defined for all
    $t \in (-\infty,-2t_0] \cup \{-t_0\} \cup [0,+\infty),$
    define $\tau(\phi(p,t))$
    for $t$ on the intervals $[-2t_0, -t_0]$ and $[-t_0, 0]$
    by linear interpolation.
    With $\tau : [0, \infty) \ti S^1 \to [T, \infty)$
    now defined, define $r : [0, \infty) \ti S^1 \to [0, \infty) \ti S^1$
    by $r(p) = \phi(p, \tau(p)).$

    We now verify that $\tau$ and $r$ have all of the desired properties.
    For an individual orbit $\phi(p, \bbR)$ with $p \in \al_2,$
    the inequality $T < 2 t_0 < \tau(p)$ implies
    that the function $\bbR \to \bbR, t \mapsto \tau(\phi(p,t))$
    is non-decreasing on the interval $(-\infty, 0]$
    and for $t \ge 0,$ it is defined as the $n$-th return time.
    Therefore, $r$ restricts to a homeomorphism on the orbit $\phi(p,\bbR).$
    As this holds for every orbit in $(0, \infty) \ti \bbR$
    and $r$ is the identity map on $0 \ti S^1,$
    it follows that $r$ is a homeomorphism of $[0, \infty) \ti \bbR.$

    In the statement of the lemma, item (1) holds by the definitions
    of $\tau$ and $r.$
    Item (2) holds (for points in $[0,\infty) \ti S^1$)
    by defining $K$ as the annulus between $0 \ti S^1$ and $\phi(\al_2, -2t_0).$
    Item (4) holds since $r(p)$ is the $n$-th return map for any $p \in A.$

    To prove item (3), first note that $\phi(\al_3, 2 t_0) = \al_1$
    and $\tau(A_3) \subof [2 t_0, \infty)$
    together imply that $r(A_3) \subof A_1.$
    The definition of $t_0$ implies that $\pi(A_1)$ and $\pi(A_3)$ are disjoint.
    Now consider a point
    \[
        q \, = \, (x,z) \, \in \, (0,1] \ti S^1
        \, \subof \, A_1 \cup A_2 \cup A_3.
    \]
    If $q \in A = A_1 \cup A_2,$
    then $r(q)$ is the $n$-th return of the forward orbit of $q$
    to the segment $[-1,1] \ti z.$
    Since $0 \ti S^1$ is the unique periodic orbit of the flow,
    $r(q) \ne q$ and so $\pi(r(q)) \ne \pi(q).$
    If instead $q \in A_3,$ then $\pi(q) \in \pi(A_3)$ and
    $\pi(r(q)) \in \pi(A_1),$ so $\pi(r(q)) \ne \pi(q).$
    This finishes the proof of item (3).

    \medskip{}

    We have completed the proof in the case where the flow on $[0,\infty) \ti S^1$
    is topologically attracting towards $0 \ti S^1.$
    Assume now that the flow $\phi$ is repelling away from $0 \ti S^1.$
    Let $\psi$ be the time-reversal of $\phi.$
    That is, $\psi(p, t) = \phi(p, -t).$
    Then the above work shows that there are functions
    $r_\psi : [0, \infty) \ti S^1 \to [0, \infty) \ti S^1$
    and
    $\tau_\psi : [0, \infty) \ti S^1 \to [T, \infty)$
    which satisfy the conclusion of the lemma
    for $\psi$ in place of $\phi.$
    As $r_\psi$ is a homeomorphism, we can define $r$ as the inverse of $r_\psi$
    and then define $\tau$ by $\tau(p) = \tau_\psi(r(p)).$
    Then $r$ and $\tau$ satisfy the conclusions of the lemma
    for the original flow $\phi.$
    This solves the problem on $[0,\infty) \ti S^1.$
    To solve it on all of $\bbR \ti S^1,$
    we do the same steps on $(-\infty,0] \ti S^1,$
    taking care to use the same large integer $n \ge 1$
    on both sides of $0 \ti S^1.$
\end{proof}    
With these lemmas established,
we now prove \cref{prop:fullgoodtime}.
As $\Fcs$ is in ideal position, the sublamination $\Lamcs$
of vertical leaves projects to a geodesic foliation on $S.$
Let $L \in \Lamcs$ be the leaf containing the unique center circle
$C = \pi \inv(p_0)$ where $p_0 \in S.$
Let $\gam^s \subof S$ be the geodesic for which $L = \pi \inv(\gam^s)$
and let $D \subof S$ be a small disk centered at $p_0$ such that
$\pi \inv(D)$ is a good neighborhood of plaques for $C.$

Let $T > 0$ be given by \cref{lemma:Ugoodtime}, with $\pi \inv(\ior(D))$ being
the neighborhood of $C$ used in that lemma.
Let $\rho : \bbR \to \gam^s$ be a parameterization of the geodesic such that
$\rho([-1,1]) = D \cap \gam^s.$
Using this, define a diffeomorphism $\rho_L : \bbR \ti S^1 \to L$ such that
\[
    \pi(\rho_L(x,z)) = \rho(x)
    \quad
    \text{ for all }
    \quad
    (x,z) \in \bbR \ti S^1.
\]
Then the flow $\phi$ on $L$ pulls back to a flow on $\bbR \ti S^1.$
Applying \cref{lemma:patching} to the pulled back flow,
we find a flow time $\tau : L \to [T, \infty)$
with the following properties:
\begin{itemize}
    \item in a neighborhood of $C,$
    $\tau$ is the $n$-th return time to the plaque,
    \item
    if $p \in L \cap \pi \inv(D),$
    then $p$ and $\phi(p, \tau(p))$ lie on the same fiber
    if and only if $p \in C,$
    \item
    outside of a compact subset of $L,$
    $\tau(p) = T.$
\end{itemize}

Since $T$ was given by \cref{lemma:Ugoodtime},
we may conclude from this that
the points $p$ and $\phi(p, \tau(p))$ lie on different fibers
for all $p \in L,$ not just those points in $\pi \inv(D) \cap L.$

Extend $\tau$ to all of $\Lamcs$ by setting $\tau(p) = T$
for all points $p \in \Lamcs \sans L.$
As $L$ is an isolated leaf,
this extension is continuous.
Let $K$ be a very small compact disk centered at $p_0$
such that the $n$-th return time to the plaque
is well defined for all points in $\pi \inv(K).$
Then use this $n$-th return time to extend
$\tau$ to a continuous function on
$\Lamcs \cup \pi \inv(K).$
Finally, choose any continuous extension to all of $M.$
The resulting function $\tau : M \to [T, \infty)$
satisfies all of the conclusions of \cref{prop:fullgoodtime}.

\section{The averaged flow revisited} \label{sec:revisit} 

Assume now that we are in the setting of \cref{prop:fullgoodtime},
where $\Fcs$ and $\Fcu$ are $C^0$ foliations with $\Fcs$ in ideal position,
every center circle has a good neighborhood of plaques,
the flow time $\tau : M \to [0, \infty)$ is given by \cref{prop:fullgoodtime},
and $f : M \to TS$ is the induced map.
In this setting, there is a finite set $\varnothing \ne Z \subof S$
such that $\pi \inv(Z) \subof M$ is the union of all center circles.
We now construct the averaged vector field,
a continuous function $X : S \to TS$ where $X(p) = 0$
if and only if $p \in Z.$

When restricted to $M \sans \pi \inv(Z),$
the induced function $f$ maps points in $M \sans \pi \inv(Z)$
to non-zero vectors in the tangent bundle of $S \sans Z.$
If we normalize these vectors,
this gives a bundle map
$f_1 : M \sans \pi \inv(Z) \to T^1(S \sans Z).$
The proof of \cite[Proposition 4.4]{ham2020horizontal}
also applies to an open manifold such as $M \sans \pi \inv(Z).$
As in the proof at the end of \cref{sec:average},
the presence of vertical cs leaves implies that
the degree of $f_1$ must be zero when restricted to a fiber.
Therefore, the averaging along fibers yields a unit vector field
$X_1 : S \sans Z \to T^1(S \sans Z).$

Define a continuous vector field on all of $S$ by
\[
    X : S \to TS,
    \quad
    p \mapsto
    \begin{cases}
        \dist(p, Z) \cdot X_1(p) & \text{if } p \notin Z, \\
        0 & \text{if } p \in Z.
    \end{cases}  \]
We now bound the indices of the fixed points of $X.$

\begin{lemma} \label{lemma:critindex}
    At each point $p \in Z,$ the index of $X$ at $p$ is in $\{-1,0,+1\}.$
    Moreover, if we consider a small disk $D$ centered at $p$ and split $D$
    along the image $\pi(L)$ of the leaf $L \in \Fcs$ through $\pi \inv(p),$
    then this produces two half-disks and 
    the index of $X$ at $p$ is in $\{ \tfrac{-1}{2} , 0 , \tfrac{+1}{2} \}$
    on each of the half-disks.
\end{lemma}   
\begin{proof}
    We can assume that $D$ here is such that $\pi \inv(D)$ gives
    the good neighbourhood of plaques
    for the center circle $C = \pi \inv(p).$
    Consider one these plaques, $D_z$ for a choice of $z \in S^1.$
    Then the induced map $f : M \to T S$ restricted to $D_z$
    defines a vector field $X_z$ on $D$ as follows.
    Recall that the good neighborhood of plaques is given by an embedding
    $i : D \ti S^1 \to M$
    where each plaque is of the form
    $D_z = i (D \ti z)$ for some $z \in S^1.$
    For each such $z,$ define a vector field $X_z$ on the disk $D$
    by
    \[
        X_z(q) = f(i(q,z)).
    \]
    For points on $D_z$ near the fiber $\pi \inv(p),$
    the flow time $\tau$ is the $n$-th return time to the plaque $D_z$ and
    the local product structure of the good neighborhood of plaques
    implies that (up to conjugating by a homeomorphism)
    the $n$-th return map is
    a product of two one-dimensional maps.
    Therefore, we are in the setting of \cref{prop:productmap}.
    (Strictly speaking, the induced map $f : M \to TS$ and therefore
    the vector field $X_z$ is defined using the hyperbolic metric
    instead of the Euclidean metric,
    but see the discussion at the end of \cref{sec:indices} to handle this.)
    By \cref{prop:productmap},
    $X_z$ has an isolated zero at $p$ with index in $\{-1, 0, +1\}.$
    By continuity, this index must be constant and independent of $z \in S^1.$

    We claim that the averaged vector field $X$ has the same index at $p.$
    This can be proved via algebraic topology.
    First identify $D$ with the unit disc in Euclidean space $\bbR^2$
    such that $S^1 = \del D.$
    Then define a map $g : S^1 \ti S^1 \to S^1$ where
    $g(q,z)$ is the unit vector pointing in the direction of $X_z(q).$
    Then $g$ lifts to a map $\tilde g : \bbR^2 \to \bbR$ on the universal cover
    where the covering map $\bbR \to S^1$ is defined by identifying
    $S^1$ with $\bbR/\bbZ.$
    That $X_z$ has an index $j \in \{-1, 0, 1\}$ means that
    $\tilde g(x + 1, y) = \tilde g(x, y) + j$
    for all $(x,y) \in \bbR^2.$
    The geodesic $\pi(L)$ cuts though the disk $D$ and therefore intersect $\del D$
    in two points.
    At either of these two points of intersection,
    the vector $X_z$ points in a constant direction
    independent of $z$ and therefore
    $\tilde g(x, y + 1) = \tilde g(x, y)$
    for all $(x,y) \in \bbR^2.$
    The averaging of angles used to produce the averaged vector field
    can be realized as a function $\tilde g_{\text{avg}} : \bbR \to \bbR$
    defined by
    \[
        \tilde g_{\text{avg}}(x) = \int_{0}^1 \tilde g(x,y) \, dy
    \]
    and therefore
    $\tilde g_{\text{avg}}(x + 1) = \tilde g_{\text{avg}}(x) + j.$
    This quotients down to a map $g_{\text{avg}} : S^1 \to S^1$ of degree $j$
    and shows that the averaged vector field $X$ has index $j$ at the point $p.$

    The same reasoning holds when we split $D$ into two half-disks.
    We can use \cref{prop:halfproductmap} to show that index of $X$
    on these half-disks lies in 
    $\{-\frac{1}{2}, 0, +\frac{1}{2}\}.$
\end{proof}
\section{Dumbbells} 

Consider a flow $\phi$ defined on a oriented surface with boundary $S_0$
and generated by a continuous vector field.
Previously, we used $\phi$ to denote the flow along the center foliation $\Fc,$
but in this section $\phi$ can be any flow generated by a continuous vector field
and the next section $\phi$ will be generated by the averaged vector
field.
As before, we use $\phi^t(x)$ and $\phi(x,t)$ interchangeably for the flow.

Assume that $S_0$ has a piecewise $C^1$ boundary where
each $C^1$ segment in the boundary is either
tangent to the flow or transverse to the flow,
and recall the definitions of inflow and outflow edges given
in \cref{sec:indices}.

We call a boundary component $C$ of $S_0$ a \emph{dumbbell} if
the flow has no fixed points on $C$ and
$C$ consists of $4n$ segments $\sig_1,\sig_2, \ldots, \sig_{4n}$ where
\begin{itemize}
    \item $\sig_i$ is tangent to the flow for $i$ odd,
    \item
    $\sig_{2+4k}$ is an outflow edge for $0 \le k < n,$ and
    \item
    $\sig_{4k}$ is an inflow edge for $0 < k \le n.$
\end{itemize}
\Cref{fig:dumbbell} shows dumbbells for $n = 1$ and $n = 2.$
The shape for $n = 1$ is what motivated the name.
We call a segment of the dumbbell a \emph{lobe} if
it is tangent to the flow.

\begin{figure}
    \centering
    \includegraphics{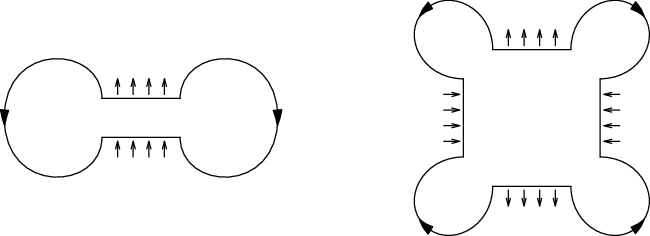}
    \caption{
    Dumbbells with $n = 1$ at left and $n = 2$ at right.
    }
    \label{fig:dumbbell}
\end{figure}

We now show how we can excise an invariant subset $\Gam$ from a surface
leaving dumbbells on the boundary.

\begin{prop} \label{prop:excise}
    Let $\phi$ be a flow tangent to a continuous vector field $X$
    on a closed oriented surface $S.$
    Suppose $\Gam$ is a compact $\phi$-invariant subset
    such that $\phi$ has no fixed points, periodic orbits, or
    isolated orbits in $\Gam.$
    Then there is a compact set $K$ containing $\Gam$ such that every boundary
    component of $S \sans \ior(K)$ is a dumbbell.

    Moreover for any given $\ep > 0,$
    the set $K$ can be constructed in such a way
    that $d_H(\Gam, K) < \ep$ in Hausdorff distance and
    all of the inflow and outflow edges have length less than $\ep.$
\end{prop}
\begin{proof}
    As $S$ is oriented, in a neighborhood $V$ of $\Gam$ we can define
    a smooth unit vector field $Y$ transverse to $X$ and
    integrate $Y$ to a get a smooth flow $\psi.$
    Let $\ep > 0$ be as in the statement of the proposition.
    By replacing $\ep$ with a smaller value if necessary,
    assume that $\psi(x, [-\ep,\ep] ) \subof V$ 
    for all $x \in \Gam.$
    Define a set $U^+ \subof \Gam$ by $x \in U^+$ if
    $\psi(x, (0,\ep) )$ intersects $\Gam.$
    Since $\Gam$ is $\phi$-invariant,
    one can see that
    $U^+$ is open in the relative topology of $\Gam \subof S$
    and therefore the set $K^+$ defined as $K^+ = \Gam \sans U^+$ is compact.

    Consider an orbit $L$ of $\phi$ which intersects $K^+.$
    Since the flow has no fixed points or periodic orbits in $\Gam,$
    the orbit $L$ must be a topological line.
    Suppose that the intersection $L \cap K^+$ contained a sequence $\{x_n\}$
    that tended to one of the two ends of $L.$
    Then $\{x_n\}$ would accumulate on a point $x \in \Gam$ and
    we could show that $x_n \in U^+$ for some large $n.$
    Therefore, $L \cap K^+$ is contained in a compact curve inside of $L.$
    A similar argument shows that 
    $K^+$ intersects only finitely many orbits of $\phi.$

    Let $L_1, \ldots, L_n$ be all of the orbits which intersect $K^+$ and
    for each $L_j,$ let $I_j$ be the smallest connected subset of $L_j$
    which contains $L_j \cap K^+.$
    Analogous to $U^+,$ define $U^-$ by $x \in U^-$ if
    $\psi(x, (-\ep,0) )$ intersects $\Gam$
    and define the complement $K^- = \Gam \sans U^-.$
    By the same logic,
    there are finitely many orbits $E_1, \ldots, E_m$ of $\phi$
    which intersect $K^-$ and we define compact curves $J_i \subof E_i$
    where $J_i$ is the smallest connected subset containing $E_i \cap K^-.$
    As $\Gam$ contains no isolated orbits of $\phi,$
    the sets $K^+$ and $K^-$ are disjoint
    and therefore all of the compact curves
    $I_1, \ldots, I_n,J_1, \ldots, J_m$ are pairwise disjoint.
    We assume for now that none of these sets is a singleton set
    and therefore each of them is homeomorphic to a compact interval.

    Now consider one of the endpoints $x_1$ of $I_1.$
    Assume this is the ``upper'' endpoint of $I_1$ so that
    $\phi^t(x_1) \in U^+$ for all $t > 0.$
    For one such $t > 0,$ let $x$ denote $\phi^t(x_1).$
    Then $x \in U^+$ implies that there is 
    an orbit of $\phi$ in $\Gam$ that intersects $\psi(x, (0,\ep) ).$
    Let $E$ be the closest such orbit;
    that is, let $s \le \ep$ be the smallest positive number such that
    $\psi^s(x) \in \Gam$ and let $E$ be the orbit of $\phi$ through $\psi^s(x).$
    If no such $E$ existed,
    then orbits would accumulate on $L_1$ from the ``positive'' side,
    contradicting the fact that $L_1 \cap K^+$ is non-empty.
    Observe that $s$ depends continuously on $t \in (0, \infty)$
    and that the orbit $E$ is independent of $t.$

    At the endpoint $x_1$ of $I_1,$ the set $\psi(x_1, (0,\ep) )$ is disjoint
    from $\Gam$ and the point $y_1$ defined by $y_1 = \psi(x_1, \ep)$ is in $\Gam.$
    Therefore, $y_1$ lies in $K^-$ and $E$ is one of the leaves
    $E_1, \ldots, E_m$ defined above.
    Up to relabelling, we may assume $E = E_1.$
    For every $r > 0,$ there are corresponding $0 < s < \ep$ and $t > 0$
    such that
    \[
        \phi^r(y_1) = \psi^s(\phi^t(x_1))
    \]
    and so $y_1$ is one of the two endpoints of $J_1.$

    The above reasoning shows that every endpoint $x$ of a segment $I_j$
    is connected to an endpoint $y$ of a segment $J_i$
    by a short curve of the form $\psi( x, [0, \ep] )$
    transverse to the flow $\phi.$
    Thus these short transverse curves along with the $I_j$ and $J_i$
    can be concatenated together into a finite number of
    piecewise $C^1$ curves $C_1, \ldots C_\ell.$
    For each curve $C_k,$ $\Gam$ accumulates on exactly one side of the curve.
    Therefore, we can define the desired set $K$ as the closure
    of those connected components of $S \sans (C_1 \cup \cdots \cup C_\ell)$
    which intersect $\Gam.$

    We now consider the more general case where one or more of the sets
    $I_1, \ldots, I_n,$ or $J_1, \ldots, J_m$ is a singleton.
    Say $I_1 = \{x_1\}$ is a singleton and that $y_1 = \psi(x_1, \ep) \in J_1.$
    Then for every $x \in L_1,$ there is $0 < s \le \ep$ such that
    $\psi(x, s) \in E_1.$
    Moreover, the equality $s = \ep$ holds exactly when $x = x_1.$
    Therefore $J_1 = \{y_1\}$ is also a singleton set
    and both $I_1$ and $J_1$ can safely removed from consideration.
    In this way, each of the singleton sets $I_j$ pairs with
    a singleton set $J_i$ and these can be removed.
\end{proof}
\section{Analysing the flow} \label{sec:finale} 

In \cref{sec:revisit}, we constructed a continuous vector field $X$
on the closed surface $S.$ Since in this setting $\Fcs$ is in ideal position,
its vertical sublamination $\Lamcs$ projects via $\pi : M \to S$ to
a geodesic lamination $\Lam \subof S.$

Together, $X$ and $\Lam$ have the following properties:
\begin{enumerate}
    \item the vector field $X$ is tangent to the geodesic lamination $\Lam;$
    \item
    the vector field is zero at a point $p$ if and only if
    $\pi \inv(p)$ is a center circle and
    we refer to these as the \emph{critical points};
    \item
    each critical point lies on an isolated geodesic in $\Lam$
    and we call these geodesics the \emph{critical geodesics};
    \item
    each critical geodesic contains exactly one critical point;
    \item
    at each critical point, the index of $X$ at $p$ is in $\{-1,0,+1\};$
    moreover, if we consider a small disk $D$ centered at $p$ and split $D$
    along the critical geodesic to produce two half-disks,
    then the index of $X$ at $p$ is in $\{ \tfrac{-1}{2} , 0 , \tfrac{+1}{2} \}$
    on each of the half-disks.
\end{enumerate}
Item (3) is given by \cref{lemma:critiso} and item (5) by \cref{lemma:critindex}.

The plan now is to first excise the non-isolated leaves of $\Lam$ from $S,$
producing a surface with boundary $S_0 \subof S$
where $\Lam \cap S_0$ consists only of isolated geodesic arcs, each of which is
compact. We then split $S_0$ into a number of ``regions'' $R_i$ by cutting along
the ``critical arcs.'' By applying the Poincar\'e--Hopf theorem to the vector
field $X$ restricted to a region $R_i,$ we show that at least one of these
regions has positive Euler characteristic and therefore must be a topological
disk. Its pre-image $\pi \inv(R_i)$ must contain half of a vertical leaf of $\Fcu,$ and
this causes a contradiction, completing the overall proof.

For simplicity, we assume that the continuous vector field $X$
integrates to a flow $\phi$ on $S$ and we explain in the following remark
how to adapt the proof when this is not the case.

\begin{remark}
    The definition of a dumbbell and the proof of \cref{prop:excise}
    do not actually rely on $X$ integrating to a flow $\phi$ on all of $S.$
    The proof of \cref{prop:excise} only considers the flow $\phi$
    on the lamination $\Gam.$
    In our setting, $\Gam$ will be a geodesic lamination and so
    the flow is well defined.
    After applying the proposition and restricting to a subset $S_0 \subof S,$
    we may replace $X$ by a smooth approximation
    such that $X$ is unchanged on the finitely many arcs of $\Lam \cap S_0$
    which contain critical points.
    Thus, we may freely assume that $X$ integrates to a flow $\phi.$
\end{remark}
\begin{lemma} \label{lemma:minimalGam}
    There exists a sublamination $\Gam$ of $\Lam$ such that
    every leaf of $\Lam$ accumulates on $\Gam$ in both the forward and backwards
    directions and
    $\Gam$ does not have any isolated leaves.
\end{lemma}
\begin{proof}
    For this, we use minimal sublaminations.
    Here, a sublamination $\Lam_1 \subof \Lam$ is minimal if every leaf
    of $\Lam_1$ is dense in $\Lam_1$.
    By Zorn's lemma, any lamination has a minimal sublamination.
    In our setting, $\Lam$ has at most finitely many minimal sublaminations.
    This can be seen, for instance, by using
    theorem 4.1 and proposition 8.5 of \cite{plante1975foliations}.
    We define $\Gam$ as the union of all minimal sublaminations of $\Lam.$
    Treating $\Lam$ as the orbits of a flow, the
    $\alpha$ and $\omega$-limit sets of each orbit are sublaminations of $\Lam$
    which contain minimal sub-sublaminations.
    Therefore any orbit accumulates on a minimal sublamination
    in both the forward and backwards directions.
\end{proof}

Now apply \cref{prop:excise} to this sublamination $\Gam$ and
let $S_0$ denote the subset $S \sans \ior(K)$
given by the proposition. Each boundary component of $S_0$ is therefore
a dumbbell.

If $L$ is an isolated leaf in $\Lam,$ then
any connected component of $L \cap S_0$ is a compact geodesic segment
and each of the two endpoints of the segment lies on either
an inflow or outflow edge of the boundary of $S_0.$
If $L$ is a critical geodesic with critical point $p,$
we call segment of $L \cap S_0$ containing $p$
a \emph{critical arc.}

Now cut $S_0$ along all of the critical arcs.
This produces a finite collection of pieces, each of which is a surface
with piecewise $C^1$ boundary.
We call each such piece a \emph{region}.

Each $C^1$ segment in the boundary of a region must be one of the following:
\begin{enumerate}
    \item a critical arc;
    \item
    a lobe of a dumbbell; or
    \item
    a subcurve of an inflow or outflow edge of a dumbbell.
\end{enumerate}
Suppose $\sig$ is an outflow edge of a dumbbell.
If none of the critical arcs intersects $\sig,$
then $\sig$ will be an outflow edge for one of the regions.
If instead, $k \ge 1$ of the critical arcs intersect $\sig,$ then
the splitting into regions will split $\sig$ into $k + 1$ subsegments
and each subsegment will be an outflow edge for one of the regions.

\begin{prop} \label{prop:eulerpositive}
    The Euler characteristic of a region $R$ satisfies
    $\chi(R) \ge \frac{n}{2}$ where $n$ is the number of lobes
    appearing in the boundary of $R.$
\end{prop}
\begin{proof}
    Consider the $C^1$ segments of a boundary component of $R.$
    These must alternate between segments tangent to the flow
    (lobes and critical arcs) and segments transverse to the flow
    (inflow and outflow edges).
    By \cref{thm:poincarehopf},
    each inflow or outflow edge contributes exactly $+\frac{1}{2}$
    to the Euler characteristic.
    If the boundary has $2m$ segments, then exactly $m$ of them are
    inflow/outflow edges.
    If the boundary has $n$ lobes, then none of these contribute to $\chi(R)$
    since they are tangent to the flow and contain no critical points.
    This leaves $m - n$ critical arcs, each of which has a single critical point
    with index $ \ge -\frac{1}{2}$. Therefore,
    \[
        \chi(R) \ \ge \ \frac{m}{2} - \frac{m-n}{2} \ = \ \frac{n}{2}
        \qedhere
    \] \end{proof}
\begin{cor} \label{cor:diskregion}
    At least one region is a topological disk
    and has a critical point on its boundary.
\end{cor}
\begin{proof}
    Consider an outflow edge $\sig$ of a dumbbell that contains at least
    one endpoint of a critical arc.
    When we split $S_0$ into regions,
    $\sig$ will be split into multiple subsegments.
    Each subsegment shares at least one endpoint with a critical arc.
    Exactly two of the subsegments (the first and last subsegments of $\sig$)
    share their other endpoint with a lobe.
    This implies that there is a region $R$ that has both a critical arc
    and a lobe on its boundary.
    Then $R$ is an oriented surface with boundary and $\chi(R) \ge \frac{1}{2},$
    so it must be a disk.
\end{proof}
This corollary, combined with the next result,
gives the needed contradiction.

\begin{lemma} \label{lemma:nodisk}
    If a region has a critical point on its boundary,
    then it is not a topological disk.
\end{lemma}
\begin{proof}
    Since $\Gam$ is the set of non-isolated leaves of
    the geodesic lamination $\Lam$ on $S,$
    its pre-image $\pi \inv(\Gam) \subof M$ is a sublamination of $\Fcs$ consisting
    of all of the non-isolated vertical leaves of $\Fcs.$

    Recall that the foliation $\Fcu$ is not in ideal position.
    In the current setting, $\Fcu = h_1(\Fcu_0)$ is a $C^0$ foliation
    with $C^0$ leaves
    where $h_1$ is given by \cref{prop:straighten}.
    By applying \cref{thm:ideal} to $\Fcu_0$ and then composing the resulting
    homeomorphism with $h_1 \inv,$
    we can find a homeomorphism $h^u : M \to M$
    isotopic to the identity and
    such that $h^u(\Fcu)$ is in ideal position.
    Let $\Lamcu$ denote the sublamination of $\Fcu$ consisting
    of vertical leaves;
    that is, $L \in \Lamcu$ if $h^u(L)$ is a union of fibers.

    The sets $\pi \inv(\Gam)$ and $\Lamcu$ cannot intersect,
    since such an intersection would produce 
    infinitely many center circles on $M.$
    Therefore on the surface $S,$ the compact subsets $\pi(\Lamcu)$ and $\Gam$
    are disjoint and are at some positive distance from each other.
    We can assume that the dumbbells given by applying \cref{prop:excise}
    have inflow and outflow edges that are much smaller than this distance
    and therefore $\pi(\Lamcu)$ does not intersect any of the inflow or outflow
    edges.

    Now consider a region $R$ and assume both that $R$ is a topological
    disk and that it has a critical point $p$ on its boundary.
    The fiber $\pi \inv(p) \subof M$ is a center circle, which we denote by $C.$
    Let $L$ denote the cu leaf $L \in \Lamcu$ that contains $C.$
    Since $C$ is the unique center circle in $L,$
    it follows that
    $L \cap \pi \inv(\del R) = C.$

    Since $R$ is simply connected, we can lift it to a topological disk
    $\hat R \subof \bbH^2$ via the universal covering map $\bbH^2 \to S.$
    Recall that $\hat M$ denotes the covering space of $M$ which is obtained
    obtained by pulling the circle bundle back by the covering $\bbH^2 \to S.$
    This covering space is homeomorphic to $\bbH^2 \ti S^1$
    and we use $\pi$ to denote the projection $\pi : \hat M \to \bbH^2.$
    We can therefore lift $C \subof M$ to a fiber $\hat C \subof \hat M$
    and $L \subof M$ to an embedded cylinder $\hat L \subof \hat M$ such that
    $\hat L \cap \pi \inv(\del \hat R) = \hat C.$
    Therefore, one of the two connected components of
    $\hat L \sans \hat C$
    is contained entirely in the compact set $\pi \inv(\hat R).$
    However, the existence of $h^u$ above means that there is a geodesic
    $\gam \subof \bbH^2$ such that $\hat L$
    lies at finite distance from $\pi \inv(\gam).$
    This implies that is $\hat L$ is properly
    embedded in $\hat M,$ which gives a contradiction.
\end{proof}
\appendix \section{Ideal foliations} \label{sec:ideal} 

This appendix gives the proofs of
\cref{thm:ideal} and \cref{addendum:superDE}
involving ideal foliations.

We start by proving a two dimensional version of \cref{thm:ideal}. Recall that a \emph{quasi-geodesic lamination} in a hyperbolic surface $\Sigma$ is a compact set $\Lambda$ which is a disjoint union of immersions $\{\gamma_x: \bbR \to \Sigma\}_{x \in \Lambda}$ with $\gamma_x(0)=x$ and with the property that when lifted to the universal cover $\tilde \Sigma \cong \bbH^2$ one has the following properties: 

\begin{enumerate}
\item there is $C>0$ so that every lift $\tilde \gamma: \bbR \to \tilde \Sigma$ of a curve in $\{\gamma_a\}_a$ is a $C$-quasi-geodesic, that is, one has that $C^{-1} |t-s| - C < d_{\tilde \Sigma}(\tilde \gamma(t), \tilde \gamma(s)) < C |t-s| + C.$. 
\item if $\gamma_1, \gamma_2$ are two (possibly the same) curves in $\{\gamma_x\}_{x \in \Lambda}$ and $\tilde \gamma_1, \tilde \gamma_2$ are lifts to $\Sigma$, then either $\tilde \gamma_1$ is a reparametrization of $\tilde \gamma_2$ or their images are disjoint.
\end{enumerate}

See \cite[Chapters 1 and 2]{calegari2007book} for more on laminations on surfaces. We will say that a quasi-geodesic lamination $\Lambda$ is \emph{redundant} if for every way to write $\Lambda$ as union of curves $\{\gamma_x\}_{x \in \Lambda}$ there are two curves $\gamma_1,\gamma_2$ which are not reparametrizations of the same curve and which have lifts $\tilde \gamma_1$ and $\tilde \gamma_2$ so that their images in $\tilde \Sigma$ lie at bounded Hausdorff distance from each other. A \emph{geodesic lamination} is a quasi-geodesic lamination for which every curve is a geodesic, since two distinct geodesics in $\bbH^2$ are unbounded distance from each other, it follows that a geodesic lamination is always non-redundant.

\begin{remark}\label{rem-param}
Given a quasi-geodesic lamination $\Lambda$ represented by a disjoint union of curves $\{\gamma_x\}_{x \in \Lambda}$ one can see that it is possible to reparametrize the curves $\gamma_x$ in order to have the following property: if $x_n \to x$ in $\Lambda$, then, $\gamma_{x_n} \to \gamma_x$ uniformly in compact sets. This is immediate if the curves are $C^1$ (in which case one can take parametrizations by arc-length), else, one can use the Morse Lemma to get geodesic representatives of each quasi-geodesic, consider the parametrization associated to the closest point projection and then do an averaging argument in order to get an injective parametrization (see e.g. \cite{fuller1965}).  We will use this idea in order to put a quasi-geodesic lamination in 'ideal' position. 
\end{remark} 

The following result is classical but we did not find a precise reference
(see \cite{eps1966curves} for similar results):

\begin{lemma}\label{lema-qglamtight}
Let $\Lambda$ be a quasi-geodesic lamination of a hyperbolic surface $\Sigma$ which is not redundant. Then, there exists a homeomorphism $h$ isotopic to identity such that $h(\Lambda)$ is a geodesic lamination. 
\end{lemma}

\begin{proof}
We first constuct a homeomorphism restricted to $\Lambda$ and then we use \cite{de1986conformally} to extend this homeomorphism to $\Sigma$. 

For this, we will work with $\pi: \tilde \Sigma \to \Sigma$ the universal cover. Since $\tilde \Sigma$ is isometric to $\bbH^2$ it has a Gromov compactification with a circle $\partial \tilde \Sigma$ which is equivariantly homeomorphic to the Gromov boundary of $\pi_1(\Sigma)$. The assumptions on $\Lambda$ imply that there is a bijection between the connected components of the lift of $\Lambda$ to the universal cover and a geodesic lamination $\hat \Lambda$ which is $\pi_1(\Sigma)$ equivariant. More precisely, for each $x \in \tilde  \Lambda$ if $\tilde \gamma_x$ is the lift of $\gamma_{\pi(x)}$ there is a unique geodesic $\ell_x$ which is bounded distance away from the image of $\tilde \gamma_x$. We can define a map $\varphi_x: \tilde \gamma_x(\bbR) \to \ell_x$ given by sending each point to the orthogonal projection onto $\ell_x$. This map is $\pi_1(\Sigma)$-equivariant bounded distance from the inclusion, and it also varies continuously as one changes the quasi-geodesic in $\tilde \Lambda$. Thus, one can use the averaging method in \cite{fuller1965} to find a $\pi_1(\Sigma)$-equivariant homeomorphism $\hat h: \tilde \Lambda \to \hat \Lambda$. See \cite[\S 8]{hp2018survey} for a detailed account. 

We claim that the resulting geodesic lamination on $\Sigma$ has finitely many complementary regions. To see this, note that each simply-connected complementary region is the image of an ideal polygon on $\bbH^2$ and has area at least $\pi$. As $\Sigma$ has finite area, there are finitely many such regions. Each non-simply-connected complementary region contains an essential curve in the surface. Taking all such curves forms a collection of disjoint essential curves, no two freely homotopic to each other, and therefore the collection is finite. As the geodesic lamination has finally many complementary regions, the same holds for the original lamination $\Lam$.
On the universal cover, there are finite orbits of components of complement of $\tilde \Lambda$ in $\tilde \Sigma$ by the fundamental group $\pi_1(\Sigma)$. 

Now we need to extend the map $\hat h$ to the whole $\tilde \Sigma$ defining a map $h: \tilde \Sigma \to \tilde \Sigma$ which coincides with $\hat h$ on $\tilde \Lambda$. 
Given a connected component $P$ of $\tilde \Sigma \setminus \tilde \Lambda$ one has a component $Q$ of $\tilde \Sigma \setminus \hat \Lambda$ associated to it via $\hat h$. We define the map $h$ from $P$ to $Q$ which extends $\hat h$ as follows: pick $(y,v) \in T^1P$ and $(z,v) \in T^1Q$ with the hyperbolic metric on each, and consider the Riemann mappings $R_{y,v}: P \to \mathbb{D}$ and $R_{z,w}: Q \to \mathbb{D}$ that map respectively $(y,v)$ and $(z,w)$ to $(0,1) \in T^1\mathbb{D}$.
When regarded as a subset of $\tilde\Sigma \cup \partial\tilde\Sigma$, the
boundary of $P$ consists of a collection of quasi-geodesics in
$\tilde\Sigma$ as well as a closed subset $K$ of the circle at infinity
$\partial\tilde\Sigma$.
Each quasi-geodesic has two endpoints in $K$ and corresponds
to a connected component of $\partial\tilde\Sigma\sans K$ with the same
endpoints. From this, we see that the boundary of $P$ is a topological
circle. The same holds for the boundary of $Q$, but with true geodesics
in place of quasi-geodesics.
Because of this, the Riemann mappings extend to the closures as homeomorphisms to the closed disk which we still denote as $R_{y,v}$ and $R_{z,w}$. We can therefore induce a map $\varphi= R_{z,w} \circ \hat h \circ R_{y,v}^{-1} : S^1 \to S^1$ where $\hat h$ is the restriction of $\hat h$ to $\partial P$ (and defined as the identity on the arcs that correspond to $\partial \tilde \Sigma$).  In \cite{de1986conformally} an extension $\Phi$ of $\varphi$ is constructed which is \emph{conformally natural}, which in our context implies that the map it defines from $P$ to $Q$ by considering $R_{z,w}^{-1} \circ \Phi \circ R_{y,v}$ does not depend on the choice of $(y,v)$ and $(z,w)$ and that it is equivariant by the action of $\pi_1(\Sigma)$ in $\tilde \Sigma$, that is, if we define $h: \tilde \Sigma \to \tilde \Sigma$ in such a way as to be defined in each $P$ as above, we get that for every $\gamma \in \pi_1(\Sigma)$ we have that $h$ restricted to $\gamma P$ is defined as $\gamma \circ h|_P \circ \gamma^{-1}$. The map $h$ is a bijection and extends to the identity on $\partial \tilde \Sigma$, so it is enough to show it is continuous to get that $h$ is the desired extension. 

Continuity follows from the fact that in $\Sigma$ there are finitely many connected components, the uniqueness properties of the definition and the continuity estimates of \cite[Lemma 2]{de1986conformally}. 
\end{proof}

We now proceed to the proof of \cref{thm:ideal}: 

\begin{proof}[Proof of \cref{thm:ideal}]
We write $\Sigma$ for the surface $S$ with a fixed hyperbolic metric that we will leave unchanged. Using Brittenham-Thurston's general position, one can assume that the foliation is given by a sublamination $\Gamma$ which consists of the preimage under $p: M \to \Sigma$ (the fiber bundle projection) of a quasi-geodesic lamination $\Lambda$, and the rest of the leaves are everywhere transverse to the fibers of $p$ (i.e. $p$ restricted to the other leaves is a submersion everywhere). 

Let $\hat \Lambda$ the geodesic lamination ensured by the previous lemma and let $h: \Sigma \to \Sigma$ be a homeomorphism homotopic to the identity mapping $\Lambda$ to $\hat \Lambda$ (note that by construction, $h$ is smooth in $\Sigma \setminus \Lambda$). We need to lift this to $M$.

In $\Sigma$ we can consider a small disk $D$ whose closure is disjoint from $\Lambda$ and $\hat D = h(D)$. Note that $M \setminus p^{-1}(D) \cong (\Sigma \setminus D) \times S^1$ and  $M \setminus p^{-1}(\hat D)  \cong (\Sigma \setminus D) \times S^1$ so one can easily lift $h$ to a map $H_0:M \setminus p^{-1}(D) \to M \setminus p^{-1}(\hat D)$ as $(x,t) \mapsto (h(x), t)$ via the identifications above. Note that this maps every leaf of $\Gamma$ to a vertical leaf obtained as the preimage by $p$ of a geodesic in $\Sigma$ and that the other leaves remain transverse to the circle fibers. Now, we need to extend $H$ to $p^{-1}(D)$ which is also identifiable with $D \times S^1$ only that the gluing corresponds to a map from $\partial D \times S^1$ to $\partial (\Sigma \setminus D) \times S^1$ which preserves fibers. So, it is enough to define an extension coinciding with the gluing map in the boundary of $D$ and which preserves fibers. This will continue to leave horizontal leaves transverse to the fibers of the circle bundle. This completes the proof.
\end{proof}

Having finished the proof of \cref{thm:ideal}, let us now prove \cref{addendum:superDE} showing that it is possible to improve the regularity close to compact parts of isolated leaves of the vertical lamination. It is likely that better regularity results hold in more generality as all the constructions enjoy a lot of flexibility, but to avoid increasing technicalities of the paper we chose to indicate only the used regularities.

\begin{addendum}\label{add-idealposition} 
Let $M$ and $\Fcal$ be as in \cref{thm:ideal}. Then, the homeomorphism $h$ produced in the theorem can be chosen so that for every compact set $K$ contained in an isolated vertical leaf, the map $h$ is $C^1$ in a neighborhood of $K$. 
\end{addendum}

Here, a vertical leaf is isolated if at each of its points it is not accumulated by other vertical leaves. 

\begin{proof}
Note that the set $K$ is contained in the preimage by $p: M \to \Sigma$ of a compact interval $I$ of an isolated leaf of $\Lambda$. One can modify the homeomorphism $h$ constructed in Lemma \ref{lema-qglamtight} in a neighborhood of $I$ in order that restricted to a neighborhood of $I$ it becomes a diffeomorphism: to do this, we use the following.

\begin{claim}
Let $\varphi: \bbR^2 \to \bbR^2$ be a homeomorphism sending $\bbR \times \{0\}$ to $\bbR \times \{0\}$. Then, given a compact interval $J \subset \bbR \times \{0\}$ and a neighborhood $U$ of $J$ there is a homeomorphism $\hat \varphi: \bbR^2 \to \bbR^2$ sending $\bbR \times \{0\}$ to $\bbR \times \{0\}$, coinciding with $\varphi$ outside $U$ and which is a diffeomorphism in a neighborhood of $J$. 
\end{claim}

This claim follows from the smoothing argument in \cite[\S 2]{hat2026kirby}.
Once this is established, the proof of \cref{thm:ideal} automatically gives the conclusion of the addendum.
\end{proof}

\bigskip

Competing interests: The authors declare none.


\bibliographystyle{alpha}
\bibliography{dynamics}

\def\cprime{$'$}
\begin{thebibliography}{CRHRHU18}

\bibitem[BBI04]{BBI1}
M.~Brin, D.~Burago, and S.~Ivanov.
\newblock On partially hyperbolic diffeomorphisms of 3-manifolds with
  commutative fundamental group.
\newblock {\em Modern dynamical systems and applications}, pages 307--312,
  2004.

\bibitem[BBI09]{BBI2}
M.~Brin, D.~Burago, and S.~Ivanov.
\newblock Dynamical coherence of partially hyperbolic diffeomorphisms of the
  3-torus.
\newblock {\em Journal of Modern Dynamics}, 3(1):1--11, 2009.

\bibitem[BFFP23]{bffp2}
Thomas Barthelm\'{e}, S\'{e}rgio~R. Fenley, Steven Frankel, and Rafael Potrie.
\newblock Partially hyperbolic diffeomorphisms homotopic to the identity in
  dimension 3, {II}: {B}ranching foliations.
\newblock {\em Geom. Topol.}, 27(8):3095--3181, 2023.

\bibitem[BFFP24]{bffp1}
Thomas Barthelm\'{e}, Sergio~R. Fenley, Steven Frankel, and Rafael Potrie.
\newblock Partially hyperbolic diffeomorphisms homotopic to the identity in
  dimension 3, {P}art {I}: {T}he dynamically coherent case.
\newblock {\em Ann. Sci. \'{E}c. Norm. Sup\'{e}r. (4)}, 57(2):293--349, 2024.

\bibitem[BFP23]{bfp2023collapsed}
Thomas Barthelm\'{e}, Sergio~R. Fenley, and Rafael Potrie.
\newblock Collapsed {A}nosov flows and self orbit equivalences.
\newblock {\em Comment. Math. Helv.}, 98(4):771--875, 2023.

\bibitem[BGHP20]{BGHP}
Christian Bonatti, Andrey Gogolev, Andy Hammerlindl, and Rafael Potrie.
\newblock Anomalous partially hyperbolic diffeomorphisms {III}: {A}bundance and
  incoherence.
\newblock {\em Geom. Topol.}, 24(4):1751--1790, 2020.

\bibitem[BI08]{BI}
D.~Burago and S.~Ivanov.
\newblock Partially hyperbolic diffeomorphisms of 3-manifolds with abelian
  fundamental groups.
\newblock {\em Journal of Modern Dynamics}, 2(4):541--580, 2008.

\bibitem[Bri93]{britt1993essential}
Mark Brittenham.
\newblock Essential laminations in {S}eifert-fibered spaces.
\newblock {\em Topology}, 32(1):61--85, 1993.

\bibitem[Bri97]{britt1997bundle}
Mark Brittenham.
\newblock Essential laminations in {$I$}-bundles.
\newblock {\em Trans. Amer. Math. Soc.}, 349(4):1463--1485, 1997.

\bibitem[Bri99]{britt1999boundary}
Mark Brittenham.
\newblock Essential laminations in {S}eifert-fibered spaces: boundary behavior.
\newblock {\em Topology Appl.}, 95(1):47--62, 1999.

\bibitem[Cal07]{calegari2007book}
Danny Calegari.
\newblock {\em Foliations and the geometry of 3-manifolds}.
\newblock Oxford Mathematical Monographs. Oxford University Press, Oxford,
  2007.

\bibitem[CRHRHU18]{crhrhu2018survey}
Pablo~D. Carrasco, Federico Rodriguez-Hertz, Jana Rodriguez-Hertz, and Ra\'{u}l
  Ures.
\newblock Partially hyperbolic dynamics in dimension three.
\newblock {\em Ergodic Theory Dynam. Systems}, 38(8):2801--2837, 2018.

\bibitem[DE86]{de1986conformally}
Adrien Douady and Clifford~J. Earle.
\newblock Conformally natural extension of homeomorphisms of the circle.
\newblock {\em Acta Math.}, 157(1-2):23--48, 1986.

\bibitem[Eps66]{eps1966curves}
D.~B.~A. Epstein.
\newblock Curves on {$2$}-manifolds and isotopies.
\newblock {\em Acta Math.}, 115:83--107, 1966.

\bibitem[FP24]{fp2024pseudo}
Sergio~R. Fenley and Rafael Potrie.
\newblock Partial hyperbolicity and pseudo-{A}nosov dynamics.
\newblock {\em Geom. Funct. Anal.}, 34(2):409--485, 2024.

\bibitem[FP26]{fp2026transverse}
S\'{e}rgio~R. Fenley and Rafael Potrie.
\newblock Transverse minimal foliations on unit tangent bundles and
  applications.
\newblock {\em Pacific J. Math.}, 343(1):39--118, 2026.

\bibitem[Ful65]{fuller1965}
F.~Brock Fuller.
\newblock On the surface of section and periodic trajectories.
\newblock {\em Amer. J. Math.}, 87:473--480, 1965.

\bibitem[Ham18]{ham2018prop}
Andy Hammerlindl.
\newblock Properties of compact center-stable submanifolds.
\newblock {\em Math. Z.}, 288(3-4):741--755, 2018.

\bibitem[Ham20]{ham2020horizontal}
Andy Hammerlindl.
\newblock Horizontal vector fields and {S}eifert fiberings.
\newblock {\em Algebr. Geom. Topol.}, 20(6):2779--2820, 2020.

\bibitem[Hat26]{hat2026kirby}
Allen Hatcher.
\newblock The {K}irby torus trick for surfaces.
\newblock {\em Enseign. Math.}, 72(1-2):161--174, 2026.

\bibitem[HH21]{hh2021surfaces}
Layne Hall and Andy Hammerlindl.
\newblock Partially hyperbolic surface endomorphisms.
\newblock {\em Ergodic Theory Dynam. Systems}, 41(1):272--282, 2021.

\bibitem[HHU11]{RHRHU-tori}
F.R. Hertz, M.A. Hertz, and R.~Ures.
\newblock Tori with hyperbolic dynamics in 3-manifolds.
\newblock {\em Journal of Modern Dynamics}, 5(1):185--202, 2011.

\bibitem[HP15]{HP2}
A.~Hammerlindl and Rafael Potrie.
\newblock Classification of partially hyperbolic diffeomorphisms in 3-manifolds
  with solvable fundamental group.
\newblock {\em J. Topol.}, 8(3):842--870, 2015.

\bibitem[HP18]{hp2018survey}
Andy Hammerlindl and Rafael Potrie.
\newblock Partial hyperbolicity and classification: a survey.
\newblock {\em Ergodic Theory Dynam. Systems}, 38(2):401--443, 2018.

\bibitem[HPS18]{hps2018seif}
Andy Hammerlindl, Rafael Potrie, and Mario Shannon.
\newblock Seifert manifolds admitting partially hyperbolic diffeomorphisms.
\newblock {\em J. Mod. Dyn.}, 12:193--222, 2018.

\bibitem[Pla75]{plante1975foliations}
J.F. Plante.
\newblock Foliations with measure preserving holonomy.
\newblock {\em Annals of Math.}, 102(2):327--361, 1975.

\bibitem[Sco83]{sco1983geometries}
Peter Scott.
\newblock The geometries of {$3$}-manifolds.
\newblock {\em Bull. London Math. Soc.}, 15(5):401--487, 1983.

\end{thebibliography}

\end{document}